\documentclass{amsart}
\usepackage{amsfonts, amssymb, amsthm, mathtools, extarrows, hyperref}
\usepackage[small,nohug,heads=LaTeX]{diagrams}
\diagramstyle[labelstyle=\scriptstyle]
\newtheorem{theorem}{Theorem}[section]
\newtheorem{lemma}[theorem]{Lemma}
\newtheorem{proposition}[theorem]{Proposition}

\theoremstyle{definition} 
 
\newtheorem{definition}[theorem]{Definition}
\newtheorem{example}[theorem]{Example}
\newtheorem{remark}[theorem]{Remark}
\usepackage{graphicx}
\usepackage{amsfonts, amssymb, mathtools, extarrows}
\usepackage[small,nohug,heads=LaTeX]{diagrams}
\diagramstyle[labelstyle=\scriptstyle]
\newcommand{\ints}{\mathbb{Z}}

\newcommand{\mfp}{\mathfrak{p}}
\newcommand{\mfm}{\mathfrak{m}}
\newcommand{\mfq}{\mathfrak{q}}
\newcommand{\im}{\text{im}}
\newcommand{\ckr}{\text{coker}}
\newcommand{\hm}{\text{Hom}}
\newcommand{\ds}{\oplus}

\newcommand{\bds}{\bigoplus}
\newcommand{\A}{\mathcal{A}}

\newcommand{\C}{\mathcal{C}}
\newcommand{\X}{\mathcal{X}}

\newcommand{\Y}{\mathcal{Y}}
\newcommand{\Om}{\Omega}
\newcommand{\ol}{\overline}
\newcommand{\ra}{\longrightarrow}
\newcommand{\proj}{\textbf{proj}\hspace{.05 cm}}
\newcommand{\Mod}{\textbf{mod}\hspace{.05 cm}}
\newcommand{\bMod}{\textbf{Mod}\hspace{.05 cm}}
\newcommand{\Add}{\textbf{add}_A\hspace{.05 cm}}
\newcommand{\Ad}{\textbf{add}_R\hspace{.05 cm}}
\newcommand{\aut}{\text{Aut}}
\newcommand{\End}{\text{End}}
\newcommand{\Eop}{E^{\text{op}}}
\newcommand{\ab}{\text{ab}}
\newcommand{\mcm}{\textbf{mcm}\hspace{.05 cm}}
\newcommand{\twohead}{\twoheadrightarrow}
\newcommand{\tail}{\rightarrowtail}

\newcommand{\vs}{\vspace{.125 cm}}
\newcommand{\diag}{\text{diag}}
\newcommand{\wt}{\widetilde}

\newcommand{\mds}{L_0\ds\cdots\ds L_t}
\newcommand{\ext}{\text{Ext}}
\newcommand{\stl}{\left\{}
\newcommand{\str}{\right\}}
\newcommand{\fp}{\text{fp}\hspace{.05 cm}}
\newcommand{\mfn}{\mathfrak n}
\newcommand{\vp}{\varphi} 
\newcommand{\sbe}{\subseteq}
\newcommand{\veps}{\varepsilon}
\author{Zachary Flores}
\address{Department of Mathematics, Colorado State University, Louis R. Weber Building, 841 Oval Drive, Fort Collins, CO 80523, USA}
\email{flores@math.colostate.edu}

\begin{document}

\title{$G$-groups of Cohen-Macaulay Rings with $n$-Cluster Tilting Objects}

\begin{abstract}
Let $(R, \mfm, k)$ denote a local Cohen-Macaulay ring such that the category of maximal Cohen-Macaulay $R$-modules $\mcm R$ contains an $n$-cluster tilting object $L$.  In this paper, we compute the Quillen $K$-group $G_1(R) := K_1(\Mod R)$ explicitly as a direct sum of a finitely generated free abelian group and an explicit quotient of $\aut_R(L)_\ab$ when $R$ is a $k$-algebra and $k$ is algebraically closed with characteristic not two.  Moreover, we compute $\aut_R(L)_\ab$ and $G_1(R)$ for certain hypersurface singularities.  
\end{abstract}

\maketitle


\section{Introduction}\label{intro}

Throughout this section $(R, \mfm, k)$ will always denote a local Noetherian ring that is Cohen-Macaulay.  Since the introduction of higher algebraic K-theory by Quillen there has been a significant effort to understand the structure of the $K$-groups $K_i(\A)$, for $\A$ an exact category.  Our particular interest is when $\A = \Mod R$, the category of finitely generated $R$-modules.  The groups $K_i(\Mod R)$ are denoted by $G_i(R)$.  They are, unsurprisingly, called the $G$-groups of $R$ (they are also called $K'$-groups in the literature and may be denoted by $K'_i(R)$).  In Section \ref{prelim}, we will discuss notation and various definitions of $K$-groups needed in the computation of $G_1(R)$.  

\vs

Let $\proj R$ be the subcategory of $\Mod R$ of finitely generated projective $R$-modules.  Now the inclusion $\proj R\hookrightarrow \Mod R$ induces a map of groups between $K_i(R) := K_i(\proj R)$ and $G_i(R)$.  It is of interest to understand the properties of this induced homomorphism.  In particular, when is this map an isomorphism?  This is precisely the case when $R$ is regular, following immediately from Quillen's Resolution Theorem (\cite{Q}, \S Theorem  3).  However, regular local rings are exceptionally well-behaved, so one cannot expect this behavior in general.  Suppose $i = 0$.  It is well-known $K_0(R)$ isomorphic to $\ints$ (see (\cite{J}, Theorem  1.3.11)), but what of $G_0(R)$?  If $R$ is regular, then $G_0(R) = \ints$.  However, if $R$ is not regular, but also has \textit{finite Cohen-Macaulay type} (that is, there are, up to isomorphism, finitely many indecomposable maximal Cohen-Macaulay $R$-modules) then the structure of $G_0(R)$ is elucidated in its entirety by the following.  

\begin{theorem}\label{k0mod}(\cite{Y}, Theorem  13.7)

Suppose there are $t$ non-free indecomposable maximal Cohen-Macaulay $R$-modules and denote by $\mathcal G$ the free abelian group on the set of isomorphism classes of indecomposable maximal Cohen-Macaulay $R$-modules.  The map $\mathcal G \ra G_0(R)$ given by $X\longmapsto  [X]$ is surjective and its kernel is generated by 
$$\left\{ X - X' - X''\hspace{.1 cm} \vert\hspace{.1 cm} \exists \hspace{.1 cm}\text{an Auslander-Reiten sequence}\hspace{.1 cm} 0\ra X'\ra X\ra X''\ra 0\right\}$$ 
And $G_0(R)\cong \ckr(\Upsilon)$, where $\Upsilon:  \ints^{\ds t}\ra \ints^{\ds(t+1)}$ is the Auslander-Reiten homomorphism.  
\end{theorem} 

The immense usefulness of Theorem \ref{k0mod} lies in the fact that the computation of $G_0(R)$ has been reduced to linear algebra, as the Auslander-Reiten homomorphism can be readily computed from the Auslander-Reiten quiver.  This quickly leads to the explicit computation of $G_0(R)$ for all simple singularities of finite type (see \cite{Y}, Proposition  13.10).  One can quickly see that these groups are often not $\ints$.   

\vs

Moving up one rung on the $K$-theory ladder, it is well-known that $K_1(R) := K_1(\proj R)\cong R^*$ (see (\cite{VS}, Example 1.6)).  However, the structure of $G_1(R)$ was not known for some time until the work of H. Holm in \cite{H} and V. Navkal in \cite{v2}.  In the former, computing $G_1(R)$ was carried out over an $R$ which has finite Cohen-Macaulay type and it was found that $G_1(R)$ could be computed as an explicit quotient of $\aut_R(M)_\ab$, with $M$ an additive generator for the category  maximal Cohen-Macaulay $R$-modules, $\mcm R$ (noting such an $M$ exists if and only if $R$ has finite Cohen-Macaulay type).  The latter produced the following.  

\begin{theorem}(\cite{v2}, Theorem  1.3)\label{viraj}

Assume that $R$ is Henselian and the category $\mcm R$ has an $n$-cluster tilting object $L$.  Let $ I$ be the set of isomorphism classes of indecomposable summands of $L$ and set $I_0 = I \backslash\left\{R \right\}$.  Then there is a long exact sequence 
$$\cdots \ra \displaystyle\bds_{L'\in I_0} G_i(\kappa_{L'})\ra G_i(\Lambda)\ra G_i(R)\ra  \displaystyle\bds_{L'\in I_0} G_{i-1}(\kappa_{L'})\ra\cdots$$
Where
$$\Lambda = \End_R(L)^{\text{op}}\hspace{.5 cm}\text{and}\hspace{.5 cm} \kappa_{L'} = \End_R(L')^{op}/\text{rad}(\End_R(L')^{op})$$

Moreover, $\kappa_{L'}$ is always a division ring, and when $R/\mfm = k$ is algebraically closed, $\kappa_{L'} = k$.  

\vs

The long exact sequence ends in presentation 
$$\displaystyle\bds_{L'\in I_0} G_0(\kappa_{L'})\ra G_0(\Lambda)\ra G_0(R)\ra 0 $$ 

of $G_0(R)$.  Since $G_0(\Lambda) = \ints^I$ and $\displaystyle\bds_{L'\in I_0} G_0(\kappa_{L'}) = \ints^{I_0}$, the presentation of $G_0(R)$ given above is precisely the one given in Theorem \ref{k0mod} when $L$ is an additive generator of $\mcm R$.

\end{theorem} 

The definition of an $n$-cluster tilting object is technical and we refer the reader to Definition \ref{ncluster} and Section \ref{sec4} for examples.  We show in Section \ref{struck1} that utilizing Theorem \ref{viraj} and techniques from \cite{H}, we can generalize and simplify the results \cite{H} on the structure of $G_1(R)$.  Keeping notation as in Theorem \ref{viraj}), our contribution in this direction is the following.  

\begin{theorem}\label{thmg1}

Let $k$ be an algebraically closed field of characteristic not $2$ and $R$ a Henselian $k$-algebra that admits a dualizing module and is also an isolated singularity.  If $\mcm R$ admits an $n$-cluster tilting object $L$ such that $\End_R(L)^\text{op}$ has finite global dimension, then there is a subgroup $\Xi$ of $\aut_R(L)_\ab$, described explicitly in Definition \ref{defsha}, and a free abelian group $\mathcal H$  such that 
$$G_1(R)\cong \mathcal H \ds  \aut_R(L)_\ab/\Xi $$
\end{theorem}
 
The utility of Theorem \ref{thmg1} is that the computation of $G_1(R)$ for some hypersurface singularities becomes tractable, as well as removing the necessity of the injectivity of the Auslander-Reiten homomorphism and the need for $R$ to have finite Cohen-Macaulay type, as required in \cite{H}.  In fact, with the long exact sequence of \cite{v2} and the machinery of \cite{H}, the proof is quite elementary.  However, before proving Theorem \ref{thmg1} in Section \ref{thmg1}, we collect the necessary details on $n$-cluster tilting objects, noncommutative algebra and functor categories in Section \ref{prelim}.  

\vs

Of course, in order to utilize Theorem \ref{thmg1}, one might want to know when $\mcm R$ admits an $n$-cluster tilting object.  This is discussed in Section \ref{sec4}.  

\vs 

The goal of explicitly computing $G_1(R)$ for specific $R$ would not be possible if we could not compute $\aut_R(L)_\ab$.  We expend some energy in Section \ref{abgroup} calculating $\aut_R(L)_\ab$ for several concrete examples.  This section and the next form the technical heart of our work.  

\vs 

Utilizing the results of Section \ref{abgroup}, we are able to explicitly compute $G_1(R)$ for several hypersurface rings in Section \ref{examples}.  See Examples \ref{ex1}, \ref{ex2}, \ref{ex3} and Proposition \ref{thmg2} for details.  

\vs

In Section \ref{discussion}, we discuss the similarities our computations share and make a conjecture.  

\vs 

We now fix notation.  We always use $A$ to denote an associative ring with identity that is not necessarily commutative; $\Mod A$ will be the category of finitely generated left $A$-modules; and $\proj A$ will be the category of finitely generated projective left $A$-modules.   

\vs

We will use the following setup:  $(R, \mfm, k)$ always denotes a commutative local Cohen-Macaulay ring such that

\vs

(a)  $R$ is Henselian.  

\vs

(b)  $R$ admits a dualizing module. 

\vs

(c)  $\mcm R$ admits an $n$-cluster tilting object.  

\vs

(d)  $R$ is an isolated singularity.  

\vs

The assumption of (a) give us that \textit{any} maximal Cohen-Macaulay module can be written uniquely as a direct sum of finitely many indecomposable maximal Cohen-Macaulay modules (see (\cite{GR}, Theorem  1.8 and Exercise 1.19)).  In fact, all of the rings for which we compute $G_1(R)$ are complete, so they already satisfy (a) (see (\cite{GR}, Corollary 1.9)).  The assumption of (b) is a standard technical assumption in representation theory of Cohen-Macaulay rings.  Currently, the assumption (c) is very much a technical black box, but we will see it is indispensable; see Definition \ref{ncluster}.  The assumption in (d) is necessary to make use of the theory of $n$-cluster tilting objects.  When necessary, we will assume that $R$ is a $k$-algebra and $\text{char}(k)\neq 2$, but we do not use this as a blanket assumption.      


\section{Preliminaries}\label{prelim}

\subsection{Some Definitions of K-groups}\label{kgroupdef}

We begin first by discussing the classical definition lower $K$-groups.  

\begin{definition} 
The \textbf{classical $K_0$-group} of $A$, denoted by $K^C_0(A)$, is defined as the Grothendieck group of the category $\proj A$.  More explicitly, choose an isomorphism class for each $P\in\proj A$ and let $X$ be the free abelian group on these isomorphism classes.  Then $K^C_0(A)$ is the quotient of $X$ by the subgroup  of $X$ generated by $\left\{ [P] - [P'] - [P''] : 0\ra P'\ra P\ra P''\ra 0\hspace{.1 cm}\text{exact}\right\}$. 

\vs

The \textbf{classical $K_1$-group} of $A$, denoted by $K^C_1(A)$, is defined as the abelianization of the infinite general linear group over $A$.  That is, using the obvious embeddings $GL_n(A)\hookrightarrow GL_{n+1}(A)$, we can form the \textit{infinite general linear group} $GL(A) := \bigcup_{n\geq 1}GL_n(A)$.  Thus $K^C_1(A)$ is $GL(A)_\ab$. 
\end{definition} 

\vs
Of principal importance in defining $K$-groups for our purposes is the following notion.  

\begin{definition} 

An \textbf{exact category} $\mathcal{Y}$ is an additive category together with a distinguished class of sequences $Y'\tail Y \twohead Y''$ called \textit{coinflations} with a fully faithful additive functor $F$ from $\mathcal{Y}$ into an abelian category $\X$ such that 

\vs

(a)  $Y'\tail Y\twohead Y''$ is a conflation in $\mathcal{Y}$ if and only if $0\ra F(Y')\ra F(Y)\ra F(Y')\ra 0$ is exact in $\X$.  

\vs

(b)  If $0\ra F(Y')\ra X\ra F(Y'')\ra 0$ is exact in $\X$, then $X\cong F(Y)$ for some $Y$ in $\mathcal{Y}$.  That is, $\mathcal{Y}$ is closed under extensions in $\X$.

\end{definition}

We note any abelian category is an exact category. Moreover, $\proj A$ is an exact category, where the conflations are taken to be the sequences that are exact in $\Mod A$.  Note that $\proj A$ is an exact category which is not abelian.  

\vs

We will need the following notions as they pertain to exact categories.  

\begin{definition}\label{deftriv}

$\mathcal{Y}$ denotes an exact category. 

\vs 

(a)  We will always work under the assumption that the objects of $\Y $ form a set.  In this regard, we say that $\mathcal{Y}$ is \textbf{skeletally small}.    

\vs

(b)  We say $\Y $ is a \textbf{semisimple exact category} if every conflation splits.  The prototypical example of a semisimple exact category is $\proj A$.  

\vs

(c)  We write $\Y_0 $ to denote $\Y $ viewed as an exact category in which the coinflations $Y'\tail Y\twohead Y''$ are such that the corresponding exact sequence in the abelian category $\X$ is split exact.  We call this the \textbf{trivial exact structure} for $\Y $.  
\end{definition} 

The definition of Bass's $K_1$ functor rests squarely upon the following notion.  

\begin{definition}\label{loop}

Let $\mathcal{Y}$ be any category.  Its \textbf{loop category} $\Omega\mathcal{Y}$ is the category whose objects are pairs $(Y, \alpha)$, $Y$ an object of $\mathcal{Y}$ and $\alpha\in \aut_{\mathcal{Y}}( Y)$.  A morphism in $\Omega\mathcal{Y}$ between two objects $(Y, \alpha)$ and $(Y', \alpha')$ is a commutative diagram in $\mathcal{Y}$ 

\begin{diagram}
Y &\rTo^f &Y'\\
\dTo_\cong^{\alpha} & &\dTo^\cong_{\alpha'}\\ 
Y &\rTo_f & Y'
\end{diagram}

\end{definition} 

\begin{remark} 

Let $\mathcal{Y}$ be a skeletally small exact category.  Its loop category $\Om\mathcal{Y}$ is also skeletally small and it is not hard to see that $\Om\mathcal{Y}$ inherits an exact structure such that $(Y', \alpha')\tail (Y, \alpha)\twohead (Y'', \alpha'')$ is a coinflation in $\Om\mathcal{Y}$ if and only if $Y'\tail Y \twohead Y''$ is a coinflation in $\mathcal{Y}$.

\end{remark} 

\begin{definition} 

Let $\mathcal{Y}$ be a skeletally small exact category and $\Om\mathcal{Y}$ be its loop category, so that $\Om\mathcal{Y}$ is also skeletally small and exact.  We define \textbf{Bass's $K_1$-group of $\mathcal{Y}$}, denoted by $K^B_1(\mathcal{Y})$, to be the Grothendieck group of $\Om\mathcal{Y}$ modulo the subgroup  generated by the following elements 
$$(Y, \alpha) + (Y, \beta) - (Y, \alpha\beta)$$
For $(Y, \alpha)$ in $\Omega\mathcal{Y}$ we denote its image in $K^B_1(\mathcal{Y})$ as $[Y, \alpha]$.  

\end{definition}

\begin{remark} 

(a) (\cite{H}, 3.4) We note for $Y\in\mathcal{Y}$, we have 
$$[Y, 1_Y] + [Y, 1_Y] = [Y, 1_Y1_Y] = [Y, 1_Y]$$
Hence $[Y, 1_Y]$ is the identity element of $K_1^B(\mathcal{Y})$.  

\vs

(b)  Unexpectedly, $K^B_1$ is a functor from the category of skeletally small exact categories to abelian groups.  Indeed, for a morphism $F$ (which is necessarily an exact functor) between $\mathcal Y$ and another skeletally small exact category, we have $K^B_1(F)([Y, \alpha]) = [F(Y), F(\alpha)]$.  

\end{remark} 

\begin{remark}\label{bassiso}(\cite{J}, Theorem  3.1.7) 

There is an isomorphism 
$$\eta_A:  K^C_1(A)\stackrel{\cong}{\ra} K_1^B(\proj A)$$ 
The isomorphism $\eta_A$ is such that $\xi\in GL_n(A)$ is mapped to the class $[A^n, \xi]\in K_1^B(\proj A)$, where elements of $A^n$ are viewed as row vectors and $\xi$ acts by multiplication on the right.  

\end{remark}  

\begin{definition} 

Let $\mathcal{Y}$ be a skeletally small exact category.  The \textbf{$i$th Quillen $K$-group of $\mathcal Y$}, denoted by $K^Q_i(\mathcal Y)$, is defined to be the abelian group $ \pi_{i+1}(BQ\mathcal{Y}, 0)$, where $Q\mathcal{Y}$ is  \textit{Quillen's Q-construction}; $BQ\mathcal{Y}$ is the classifying space of $Q\mathcal{Y}$; $0$ is a fixed zero object; and $\pi_{i+1}$ denotes the taking of a homotopy group.  

\end{definition} 

By (\cite{Q},Section 2, Theorem  1) there is a natural isomorphism of between the Grothendieck group functor and $K^Q_0$ (as functors on the category of skeletally small exact categories).  Moreover, $K^Q_1(\proj A)$ is naturally isomorphic to $ K^C_1(A)$ (see (\cite{VS}, Corollary 2.6 and Theorem  5.1)).  Quillen's definition of higher $K$-theory is stunningly elegant, but does not often lend itself to performing computations with ease.  The definition of Bass's functor $K^B_1$ will be more suited for our computational needs and, we will want to exploit this in the sequel.  As in (\cite{H}, 3.6), we will make strong use of the following theorem.  

\begin{theorem}\label{gersher}
There exists a natural transformation $\zeta:  K_1^B\ra K_1^Q$, which we call the Gersten-Sherman transformation, of functors on the category of skeletally small exact categories such that $\zeta_{\mathcal{Y}}:  K_1^B(\mathcal{Y})\ra K^Q_1(\mathcal{Y})$ is an isomorphism for every semisimple exact category $\Y$.  In particular, $\zeta_{\proj A}:  K_1^B(\proj A)\ra K_1^Q(\proj A)$ is an isomorphism for every ring $A$.  
\end{theorem} 

The name for $\zeta$ was introduced in \cite{H} for the following:  The existence of $\zeta$ was initially sketched by Gersten in (\cite{ger}, sect. 5) and the details were later filled in by Sherman (\cite{sher}, sect. 4), whom also proved $\zeta_\mathcal{Y}$ is an isomorphism for every semisimple exact category.    

\subsection{$n$-Auslander-Reiten Theory}\label{nartt}
We want to discuss generalizations of Auslander-Reiten theory, following \cite{i}.  To do so, we will require some precise categorical language.  Here $\Y$ denotes any exact category.  

\begin{definition} 
Write $\bMod \Y$ for the category of additive contravariant functors $\Y \ra \textbf{Ab}$, with $\textbf{Ab}$ the category of abelian groups.  The morphisms in $\bMod\Y $ are natural transformations between functors with kernels and cokernels computed pointwise.  An easy check shows that $\bMod\Y$ is abelian.  We write $(\bullet, Y) $ to denote the additive contravariant functor $\hm_\Y(\bullet, Y)$.  We say $F\in\bMod\Y$ is \textbf{finitely presented} if there is an exact sequence
$$(\bullet, Y)\ra (\bullet, Y')\ra F\ra 0$$
in $\bMod\Y$.  We write $\Mod\Y$ for the subcategory of finitely presented functors.  
\end{definition} 

For a ring $A$, let $\bMod A$ denote the category of all left $A$-modules and denote the subcategory of finitely presented left $A$-modules by $\Mod_{\fp} A$.  Fix a left $A$-module $N$ and denote by $E$ its endomorphism ring $\End_A(N)$.  Then $N$ has a left $E$-module structure that is compatible with its left $A$-module structure such that for $e\in E$ and $n\in N$, $e\cdot n  = e(n)$.  Denote by $\Add N$ the category of $A$-modules that consists of all direct summands of finite direct sums of $N$.  For $F\in \bMod(\Add N)$, the aforementioned left $E$-module structure on $N$ induces a  left-$\Eop$-module structure on the abelian group $FN$ such that $e\cdot z = (Fe)(z)$ for $e\in \Eop$ and $z\in FN$.  We use these facts for the following proposition, which will be essential in the proof of Theorem \ref{thmg1}.  

\begin{proposition} \label{emequiv}  (\cite{H}, Proposition  6.2)

There are quasi-inverse equivalences of abelian categories 
\begin{diagram}
\bMod(\Add N)&&&\pile{\rTo^{e_N}\\ \quad\simeq\quad \\ \lTo_{f_N}} &&& \bMod\Eop
\end{diagram}
Where the functors $e_N$ and $f_N$ are defined as follows:  $e_N(F) = FN$ (evaluation) and $f_N(Z) = Z\otimes_E \hm_A(\bullet, N)\vert_{\Add N}$ (functorification).  Also, these quasi-inverse equivalences restrict to equivalences between categories of finitely presented objects 
\begin{diagram}
\Mod(\Add N) &&& \pile{\rTo^{e_N}\\ \quad\simeq\quad \\ \lTo_{f_N}} &&& \Mod_\fp\Eop
\end{diagram}
\end{proposition}
\begin{definition}

Let $\X$ be an additive category and $\C$ a subcategory of $\X$.  We call $\C$ \textbf{contravariantly finite}, if for any $X\in \X$ there is a morphism $f:  C\ra X$ with $C\in \C$ such that 
$$(\bullet, \C)\stackrel{\bullet f}{\ra} (\bullet, X)\ra 0$$
is exact (where $\bullet f$ is the map induced by $f$).  Such an $f$ is called a \textbf{right}-$\C$-\textbf{approximation} of $X$.  We dually define a \textbf{covariantly finite} subcategory and a \textbf{left}-$\C$-\textbf{approximation}.    A contarvariantly and covariantly finite subcategory is called \textbf{functorially finite}.  

\end{definition} 

At long last, we are able to define an $n$-cluster tilting object.    

\begin{definition}\label{ncluster}
Let $\Y$  be an exact category with enough projectives.  For objects $X, Y$ in $\Y$ we write $X \perp_n Y$ if $\ext^i_\Y(X, Y) = 0$ for $0 < i \leq n$.  For an exact subcategory $\C \subset\Y$, we put 
$$\C^{\perp_n} = \stl X\in \Y : Y \perp_n X\hspace{.1 cm}\text{for all}\hspace{.1 cm} Y\in \C \str$$
$$^{\perp_n}C= \stl X\in \Y :  X \perp_n Y \hspace{.1 cm}\text{for all}\hspace{.1 cm} Y\in \C \str$$
We call $\C$ an $n$-\textbf{cluster-tilting} subcategory of $\Y$ if it is functorially finite and $\C = \C^{\perp_{n-1}} = ^{\perp_{n-1}}C$.  An object $L$ of $\Y$ is called $n$-\textbf{cluster-tilting} if $\text{add}_\Y(L)$ is an $n$-cluster tilting subcategory.   

\end{definition} 

From the definition of $n$-cluster tilting, if $\mcm R$ admits an $n$-cluster tilting object $L$, then $R$ is necessarily a direct summand of $L$.  While the definition of $n$-cluster tilting is quite a bit to digest at once, there are concrete examples of $n$-cluster tilting objects over familiar rings and we refer the reader to Section \ref{sec4} for several examples.  

\vs 

When $R$ has finite Cohen-Macaulay type, we have the classical notion of an \textit{Auslander-Reiten sequence} or \textit{almost-split sequence}.  When $\mcm R$ has an $n$-cluster tilting subcategory, we have the following generalization.  

\begin{definition}\label{narhom}

If $\C\subset \mcm R$ is an $n$-cluster tilting subcategory, given $X\in \mcm R$ not free and indecomposable, an exact sequence 
$$0\longrightarrow C_n\stackrel{f_n}{\longrightarrow}\cdots\stackrel{f_1}{\longrightarrow} C_0\stackrel{f_0}{\longrightarrow} X\longrightarrow 0$$ 
with $C_0,\ldots, C_n\in \C$ such that 
$$0\longrightarrow (\bullet, C_{n-1})\stackrel{\bullet f_n}{\longrightarrow}\cdots \stackrel{\bullet f_1}{\longrightarrow} (\bullet ,C_0)\stackrel{\bullet f_0}{\longrightarrow} (\bullet, X)\longrightarrow 0$$ 
is a minimal projective resolution of $(\bullet, X)/\textbf{rad}_{\mcm R}(\bullet, X)$ in $\Mod \C$ is called an $n$-\textbf{Auslander-Reiten sequence} (or an $n$-\textbf{almost-split sequence}).  

\vs 

Here $\textbf{rad}_{\mcm R}(\bullet, X)$ is such that 
$$\textbf{rad}_{\mcm R}(Y, X)  = \stl f\in \hm_R(Y, X) : fg\in \text{rad}(\End_R(Y))\hspace{.1 cm}\text{for all}\hspace{.1 cm} g\in \hm_R(Y, X)\str$$ 
If $\C\subset \mcm R$ is an $n$-cluster tilting subcategory, then $n$-Auslander-Reiten sequences exist by (\cite{i2}, Theorem 3.31).  

\end{definition}

\subsection{Endomorphism Rings and $K$-groups}\label{endokgroup}

By our blanket assumptions on $R$, there is a unique decomposition of the $n$-cluster tilting object $L = L_0^{\ds l_0}\ds \cdots\ds L_t^{\ds l_t}$, such that $L_i\in \mcm R$ is indecomposable and $l_i > 0$ and the $L_i$ are pairwise non-isomorphic.  In this section, we will assume that $l_i = 1$.  For if we write $L_{\text{red}} = L_0\ds\cdots\ds L_t$, then $\Ad L = \Ad L_{\text{red}}$.  Thus $L$ is an $n$-cluster tilting object for $\mcm R$ if and only if $L_{\text{red}}$ is.  Moreover, we will see in Section \ref{struck1}, that in the context of Theorem \ref{thmg1}, the choice of $L_\text{red}$ over $L$ is immaterial.  Write $\C = \Ad L$.  The following construction is from (\cite{H}, Construction 2.6).  

\vs 

If $L'\in \C$, we can write $L' = L_0^{\ds m_0}\ds \cdots\ds L_t^{\ds m_t}$ for uniquely determined $m_0,\ldots, m_t \geq 0$.  Set $q = q(L') = \max\left\{m_0,\ldots, m_t\right\}$ and $v_j = v_j(L') = q - m_j$.  Notice that $q$ is the smallest integer such that $L'$ is a direct summand of $L^{\ds q}$.  Now form the $R$-module $L'' = L_0^{\ds v_0}\ds\cdots\ds L_t^{\ds v_t}$ and let $\psi:  L'\ds L'' \ra L^{\ds q}$ be the $R$-linear isomorphism that takes the element 
$$((\underline{x}_0,\ldots, \underline{x}_t), (\underline{y}_0,\ldots, \underline{y}_t)\in L'\ds L'' = (L_0^{\ds m_0}\cdots\ds L_t^{\ds m_t}) \ds (L_0^{\ds v_0}\ds \cdots \ds L_t^{\ds v_t})$$
where $\underline{x}_j\in L_j^{\ds m_j}$ and $\underline{y}_j\in L_j^{\ds v_j}$, to the element 
$$((z_{01},\ldots, z_{t1}),\ldots, (z_{0q},\ldots, z_{tq}))\in L^{\ds q} = (L_0\ds\cdots\ds L_t)^{\ds q}$$ 
with $z_{j1},\ldots, z_{jq}\in L_j$ given by 
$$(z_{j1},\ldots, z_{jq}) = (\underline{x}_j, \underline{y}_j)\in L_j^{\ds q}  = L_j^{\ds(m_j + v_j)}$$ 
Now for $\alpha\in\aut_R(L')$, we define $\wt{\alpha}$ to be the automorphism on $L^{\ds q}$ given by $\psi(\alpha \ds  1_{L''})\psi^{-1}$.  Note that $\wt{\alpha} = (\wt{\alpha_{ij}})$, with $\wt{\alpha_{ij}}$ uniquely determined endomorphisms of $L$.  In particular, $\wt{\alpha}\in\mathbb M_q(\End_R L)$.  As in \cite{H}, we refer to this construction as the \textbf{tilde construction}.  

\begin{remark}\label{situations}

We note a special case of the tilde construction.  Keep notation as above.  Suppose $\alpha = a 1_{L'}$ with $a\in R^*$.  If $L' = L_{i_1}^{\ds q}\ds\cdots\ds L_{i_h}^{\ds q}$ with $0\leq i_1 < i_2 < \cdots < i_h\leq t$.  Then $\wt\alpha:  L^{\ds q} \ra L^{\ds q}$ is the automorphism given by $e1_{L^{\ds q}}$ with $e\in \aut_R(L)$ given by 
$$\diag(1_{L_0},\ldots, a1_{L_{i_1}},\ldots, a1_{L_{i_h}},\ldots, 1_{L_t})$$
Hence, $(\wt{a1_{L'}})^{-1} = \wt{a^{-1}1_{L'}}$. 

\end{remark}

As we will often be working explicitly with highly noncommutative rings, we need to discuss important ideas at the intersection of noncommutative algebra and $K$-theory.  Let $J(A)$ be the Jacobson radical of the not necessarily commutative ring $A$.  Recall that $A$ is said to be \textit{semilocal} if $A/J(A)$ is semisimple.  That is, every left $A/J(A)$-module has the property that each of its submodules is a direct summand of $A/J(A)$.  In the case that $A$ is commutative, this is equivalent to $A$ having only finitely many maximal ideals (\cite{Lam}, Proposition  20.2).  Of great importance to us is the following situation:  If $A$ is a commutative semilocal Noetherian ring and $N$ is a nonzero finitely generated $A$-module, then $\End_A(N)$ is semilocal in the preceding sense (\cite{H}, Lemma 5.1). We will see how the following remark utilizes this small but essential fact in the proof of Theorem \ref{thmg1}.  

\begin{remark}\label{vas} (\cite{H}, Paragraph 5.2)

For arbitrary $A$, denote the composition of the following group homomorphisms 
$$A^* = GL_1(A) \hookrightarrow GL(A)\twohead GL(A)_\ab = K^C_1(A)$$ 
by $\vartheta_A$.  Since $K^C_1(A)$ is abelian, there is an induced map $\theta_A:  A^*_\ab\ra K^C_1(A)$.  If $A$ is semilocal, then (\cite{B}, V\S9 Theorem  9.1) shows that $\vartheta_A$ is surjective, hence so is $\theta_A$.  When $A$ contains a field $k$ with $\text{char}(k)\neq 2$, a result of Vaserstein (\cite{L}, Theorem  2)  shows that $\theta_A$ is an isomorphism.  In particular, if $R$ is a $k$-algebra, $\text{char}(k)\neq 2$ and $M$ is a finitely generated $R$-module with $E = \End_R(M)$, then $\theta_E$ and $\theta_{\Eop}$ are isomorphisms.  

\vs

Suppose now $A$ is a commutative semilocal ring, so that the commutator subgroup $[A^*, A^*]$, is trivial, hence $\theta_A:  A^*\ra K^C_1(A)$ is surjective.  In (\cite{H}, Remark. 5.4), if $\theta_A$ is an isomorphism, an explicit inverse to $\theta_A$ is constructed:  The determinant homomorphisms $\text{det}_n:  GL_n(A)\ra A^*$ induce a homomorphism $\text{det}_A:  K^C_1(A)\ra A^*$ (since each $\text{det}_n$ is trivial on commutators in $GL(A)$) which satisfies $\text{det}_A\theta_A = 1_{A^*}$, so that $\theta_A^{-1} = \text{det}_A$.  

\end{remark}

Using Remark \ref{vas} as motivation, the following definition is made in \cite{H}.   

\begin{definition}\label{gendet} 

Let $A$ be a ring for which the map $\theta_A:  A^*_\ab\ra K^C_1(A)$ is an isomorphism.  The inverse $\theta_A^{-1}$ is denoted by $\text{det}_A$ and is is called the \textbf{generalized determinant}.  

\end{definition} 

The following proposition makes use of the tilde construction and will be useful in proving Theorem \ref{thmg1}.  We note it is essentially proven in \cite{H}, where it is a synthesis of (\cite{H}, Lemma 6.5) and the proof of (\cite{H}, Proposition 8.8).  We also note that the assumptions in (\cite{H}, Proposition 8.8) are that $R$ has finite Cohen-Macaulay type.  However, we note that under our assumptions, the portion of the proof we are referencing (\cite{H}, equation (8.8.1)) still holds.

\begin{proposition}\label{tildeprop} 
Keeping our general assumptions, suppose in addition that $R$ is an algebra over its residue field $k$ and the characteristic of $k$ is not two.  Let $L_0,\ldots, L_t\in \mcm R$ and $L$ be their direct sum.  Set $\Lambda = \End_R(L)^\text{op}$.  Let $\mathcal C_0 = \Ad (L)$ be equipped with the trivial exact structure.  If $\Lambda$ has finite global dimension, then there is an isomorphism of groups 
$$\tau:  K_1^B(\mathcal C_0) \ra \aut_R(L)_\ab$$ 
such that for any $L'\in \mathcal C_0$ and any $\alpha\in\aut_R(L')$, $\tau([L', \alpha]) = \det_{\Lambda^\text{op}}(\widetilde \alpha)$.  

\end{proposition} 

\begin{remark}\label{rho}(\cite{H}, Observation  8.9)

Let $A$ be any commutative Noetherian local ring and $\eta_A$ be the isomorphism from Remark \ref{bassiso} and $\theta_A:  A^*\ra K^C_1(A)$ be the induced map from Remark \ref{vas}.  Then $\theta_A$ is an isomorphism by (\cite{VS}, Example 1.6).  Thus the composition $\rho_A = \eta_A\theta_A:  A^*\ra K^B_1(\proj A)$ is an isomorphism such that $a\in A^*$ is mapped to $[A, a1_A]$.  

\end{remark} 

We now combine the the above preliminaries with the tilde construction to define the subgroup $\Xi$ of $\aut_R(L)_\ab$ in Theorem \ref{thmg1}.  

\begin{definition}\label{defsha} 
Recall that we are assuming that $\mcm R$ has $n$-cluster-tilting object of the form $L = L_0\ds\cdots\ds L_t$.  We assume that $L_0 = R$ and that for $j > 0$, the $L_j$ are non-free pairwise non-isomorphic and indecompsable objects in $\mcm R$.  Suppose also that $R$ is a $k$-algebra, $\text{char}(k)  \neq 2$ and $k$ is algebraically closed.  If $\mcm R$ has an $n$-cluster tilting object $L$ such that $\Lambda := \End_R(L)^\text{op}$ has finite global dimension, we define a subgroup $\Xi$ of $\aut_R(L)_\ab$ as follows: For $j > 0$, let 
$$0\ra C_n^j\ra \cdots \ra C_0^j\ra L_j\ra 0$$ 
be the $n$-Auslander-Reiten sequence ending in $L_j$ (see Definition \ref{narhom}).  By Remark \ref{vas}, $\theta_{\Lambda^\text{op}}:  \aut_R(L)_\ab \ra K^C_1(\Lambda^\text{op})$ is an isomorphism with inverse given by $\det_{\Lambda^\text{op}}$.  Then $\Xi$ is the subgroup generated by the elements given by 
$$\widetilde{a1_{L_j}}\prod_{i=1}^{n+1}\text{det}_{\Lambda^{\text{op}}}(\widetilde{a1_{C^j_{i-1}}})^{(-1)^i}$$ 
where $a$ runs over all elements of $k^*$ and $j = 1,\ldots, t$.  

\end{definition}


\section{The Structure of $G_1(R)$}\label{struck1}

In this section, unadorned $K$-groups are the Quillen $K$-groups.  Our goal of this section is to prove Theorem \ref{thmg1}.  We always assume that $\mcm R$ has an $n$-cluster tilting object $L = L_0^{l_0}\ds \cdots \ds L_t^{\ds l_t}$, with $L_0 = R$, and $L_1,\ldots, L_t$ non-free, non-isomorphic indecomposable maximal Cohen-Macaulay $R$-modules such that $\Lambda := \End_R(L)^{\text{op}}$ has finite global dimension.  In addition to our blanket assumptions, we assume that $k$ is algebraically closed of characteristic not two and $R$ is a $k$-algebra.  We begin with an easy reduction. 

\begin{lemma}\label{morita}

Set $L_\text{red} = L_0 \ds \cdots\ds L_t$.  If $\Lambda_{\text{red}} = \End_R(L_{\text{red}})^{\text{op}}$, then $\Lambda$ and $\Lambda_{\text{red}}$ are Morita-equivalent.  In particular, $G_i(\Lambda)\cong G_i(\Lambda_{\text{red}})$ for all $i\geq 0$.  

\end{lemma} 

\begin{proof} 

The desired Morita equivalence is from (\cite{dfi2}, Lemma 2.2).  Thus the categories of left $\Lambda$ and $\Lambda_{\text{red}}$ modules are equivalent, hence there is an equivalence of exact categories between $\Mod \Lambda $ and $\Mod \Lambda_{\text{red}}$.  It is well-known this yields an isomorphism in $G$-theory, hence $G_i(\Lambda)\cong G_i(\Lambda_{\text{red}})$ for all $i\geq 0$.  

\end{proof}

It is easy to see $\Ad L = \Ad L_{\text{red}}$.  Moreover, since $\Lambda$ has finite global dimension, the Morita equivalence of Lemma \ref{morita} gives that $\Lambda_\text{red}$ also has has finite global dimension.  Since $\Lambda$ has finite global dimension and is a semilocal algebra over a field of characteristic not two, by Quillen's Resolution Theorem (\cite{Q}, \S Theorem  3), (\cite{VS}, Corollary 2.6 and Theorem  5.1), and (\cite{L}, Theorem  2) we have isomorphisms 
$$G_1(\Lambda) \cong K_1(\Lambda) \cong K^C_1(\Lambda) = \Lambda^*_\ab = \aut_R(L)_\ab $$ 
As noted above, $\Lambda_\text{red}$ has finite global dimension, hence the same arguments apply, so that the above remarks and Lemma \ref{morita} give 
$$\aut_R(L_\text{red})_\ab = (\Lambda_\text{red})^*_\ab \cong G_1(\Lambda_\text{red}) \cong G_1(\Lambda) \cong \Lambda^*_\ab = \aut_R(L)_\ab$$ 
Thus may safely assume that the $n$-cluster tilting object $L$ for $\mcm R$ has the form $L_0\ds\cdots \ds L_t$, where the $L_i$ are non-isomorphic indecomposable maximal Cohen-Macaulay.  Henceforth, we always use $\Lambda$ to denote $\End_R(L)^{\text{op}}$ with $L = L_0\ds\cdots \ds L_t$, $L_0 = R$ and for $j > 0$, the $L_j$ are non-free, non-isomorphic indecomposable objects in $\mcm R$.  

\vs 

Since $k$ is algebraically closed, $\kappa_{L_j} =\End_R(L_j)^\text{op}/\text{rad}(\End_R(L_j)^\text{op}) = k$ for all $j$ (this is essentially Nakayamma's lemma).  By Theorem \ref{viraj}, there is an exact sequence of abelian groups  
\begin{equation*} G_1(k)^{\ds t} \stackrel{\gamma}{\ra} G_1(\Lambda )\ra G_1(R)\ra G_0(k)^{\ds t} \ra  G_0(\Lambda)\ra G_0(R) \ra 0\end{equation*}
By Theorem \ref{viraj}, $G_0(\Lambda) = \ints^{\ds (t+1)}$.  Moreover, is well-known that $G_0(k) = \ints$.  In particular, the above exact sequence becomes 
\begin{equation}G_1(k)^{\ds t}  \stackrel{\gamma}{\ra} G_1(\Lambda )\ra G_1(R)\ra \mathcal H \ra 0\tag{$\star$}\end{equation}
where $\mathcal H$ is the kernel of a map $\ints^{\ds t}\ra \ints^{\ds (t+1)}$.  Now $\mathcal H$ is free, being the subgroup of a free group, hence the exactness of $(\star)$ gives an isomorphism 
$$G_1(R)\cong \ckr(\gamma)\ds \mathcal H$$ 
Thus to prove Theorem \ref{thmg1}, that is, in order to calculate $\Xi$, we need to explicitly describe the map $\gamma$.  In this direction, we first define $\mathcal C_0$ to be the category $\mathcal C := \Ad L = \Ad(L_0, \ldots, L_t)$ equipped with trivial exact structure.  As we are assuming $\Lambda$ has finite global dimension, (\cite{H}, Lemma 6.5) gives an isomorphism $K_1(\mathcal C_0)\cong K_1(\Mod \mathcal C)$ that is induced by the exact Yoneda functor $y_L:  \mathcal C_0 \ra \Mod \mathcal C$, where $y_L(X) = \hm_R(\bullet, X)|_{\mathcal C}$.  Since $\Lambda$ is left Noetherian, Proposition \ref{emequiv} gives that the evaluation functor $e_L:  \Mod\mathcal C\ra \Mod \Lambda$ is an equivalence, hence induces an isomorphism $K_1(\Mod\mathcal C)\cong K_1(\Lambda)$.  Moreover, $\Lambda$ has finite global dimension, so that Quillen's Resolution Theorem (\cite{Q}, \S Theorem  3) yields that the inclusion functor $\proj \Lambda \ra \Mod \Lambda$ induces an isomorphism $K_1(\Lambda) \cong G_1(\Lambda)$.  Hence there is a map $\alpha:  G_1(k)^{\ds t}\ra K_1(\mathcal C_0)$ such that the diagram
 \begin{diagram}[tight,width=3em,height=3em]
& & K_1(\mathcal C_0)\\
&\ruTo(2, 2)^{\alpha} & \dTo_\cong \\
G_1(k)^{\ds t}  & \rTo_{\gamma} & G_1(\Lambda)
\end{diagram}
commutes.  This gives $\ckr(\gamma)\cong\ckr(\alpha)$.  Thus to prove Theorem \ref{thmg1}, it suffices to compute $\ckr(\alpha)$.  In fact, $\alpha$ is computed in the discussion of (\cite{v2}, Section 7.2).  The details will be useful and we recall them.  Now $L = \mds$, with $L_0 = R$ and $L_1,\ldots, L_t$ are the non-free indecomposable and non-isomorphic summands of $L$.  We set $I = \stl L_0,\ldots, L_t\str$ and $I_0 =  I\backslash \stl R\str$.  For $j > 0$ let 
$$0\ra C_n^j\ra \cdots \ra C_0^j\ra L_j\ra 0$$ 
be the $n$-Auslander-Reiten sequence ending in $L_j$ (see Definition \ref{narhom}).  Denote by $k_j$ the object of $\ds_{I_0} \Mod k$ which is $k$ in the $L_j$-coordinate and $0$ in the others.  We remark that to define a $k$-linear functor out of $\ds_{I_0} \Mod k$, one needs only to specify the image of each object $k_j$.  We define $k$-linear functors 
$$a_i:  \bds_{I_0} \Mod k \ra \C_0\hspace{1 cm} (0\leq i\leq n+1)$$ 
by 
$$\left\{
     \begin{array}{lr}
       a_i(k_j)  = C^j_{i-1} & (1\leq i\leq n+1)\\
       a_0(k_j)  = L_j    
     \end{array}
   \right.$$
It is shown in (\cite{v2}, Section 7.2) that $\alpha = \sum_{i=0}^{n+1}(-1)^iK_1(a_i)$.  We have the following.  
\begin{proposition}\label{mainprop} 
If $\Xi$ is the subgroup of $\Lambda^*_\ab$ from Definition \ref{defsha}, there is an isomorphism $\ckr(\alpha)\cong \Lambda^*_\ab/\Xi$.  
 \end{proposition} 
Now Proposition \ref{mainprop} implies Theorem \ref{thmg1}, so the proof of Proposition \ref{mainprop} will conclude this section.  

\begin{proof}

Since the morphisms $a_i:  \bds_{I_0} \Mod k\ra \C_0$ are functors on exact categories, they also define maps $K^B_1(a_i):  K^B_1(\bds_{I_0} \Mod k) \ra K^B_1(\C_0) $ on the Bass $K_1$-groups.  Now $|I_0| = t$, so that $K^B_1(\bds_{I_0} \Mod k) = \bds_{I_0} K_1^B(\Mod k) =  K^B_1(\Mod k)^{\ds t}$.  Let $\beta:  K^B_1(\Mod k)^{\ds t}\ra K^B_1(\C_0)$ be the map given by $\sum_{i=0}^{n+1}(-1)^iK^B_1(a_i)$.  Our first task is to show that $\ckr(\alpha)\cong\ckr(\beta)$.  The Gersten-Sherman transformation (see Theorem \ref{gersher}) $\zeta:  K^B_1\ra K_1$ provides the following commutative diagram for $i = 0, 1,\ldots, n+1$
\begin{diagram}
K_1(\Mod k)^{\ds t}&\rTo^{K_1(a_i)} & K_1(\C_0)\\
\dTo^{\zeta_{\Mod k}^{\ds t}}_\cong& &\dTo_{\zeta_{\C_0}}^\cong\\ 
K^B_1(\Mod k)^{\ds t}  &\rTo_{K^B_1(a_i)} & K^B_1(\C_0)
\end{diagram}
Where the vertical isomorphisms come courtesy of Theorem \ref{gersher}, as $\C_0$ and $\Mod k$ are semisimple exact categories.  Hence there is a commutative diagram 
\begin{diagram}
K_1(\Mod k)^{\ds t} &&\rTo^{\alpha}& & K_1(\C_0)\\
\dTo^{\zeta_{\Mod k}^{\ds t}}_\cong&&&&  \dTo_{\zeta_{\C_0}}^\cong\\ 
K^B_1(\Mod k)^{\ds t} &&\rTo_{\beta} && K^B_1(\C_0)
\end{diagram}
This gives that $\ckr(\alpha)\cong \ckr(\beta)$.  To finish the proof, first note that Remark \ref{rho} furnishes an isomorphism $\rho_k:  k^* \ra  K^B_1(\Mod k)$ such that $a\mapsto [k, a1_k]$, hence there is an isomorphism $\rho_k^{\ds t}:  (k^*)^{\ds t}\ra K^B_1(\Mod k)^{\ds t}$.  Now recall the isomorphism $\tau:  K_1^B(\mathcal C_0) \ra \Lambda^*_\ab$ (noting $\Lambda^*_\ab = \aut_R(L)_\ab$) of Proposition \ref{tildeprop}.  The map $\tau$ is such that for $L'\in \C_0$ and any $f\in\aut_R(L')$, $\tau([L', f]) = \det_{\Lambda^\text{op}}(\widetilde f)$, where $\det_{\Lambda^\text{op}}$ is the generalized determinant of Definition \ref{gendet} and $\wt f\in\aut_R(L)$ is the map obtained from the tilde construction of Subsection \ref{endokgroup}.  In particular, $\ckr(\beta)\cong \ckr(\tau\beta\rho_k^{\ds t})$, hence we calculate the latter.  Restricting to the $j$th coordinate of $(k^*)^{\ds t}$, by slight abuse of notation, we have for $a\in k^*$ 
$$\beta\rho_k(a) = \beta([k, a1_k]) = [L_j, a1_{L_j}] + \sum_{i=1}^{n+1}(-1)^i[C^j_{i-1}, a1_{C^j_{i-1}}]$$ 
By definition, $\det_{\Lambda^\text{op}}(\wt{a1_{L_j}}) =  \wt{a1_{L_j}} $, so that 
$$\tau\beta\rho_k(a) = \tau\left([L_j, a1_{L_j}] + \sum_{i=1}^{n+1}(-1)^i[C^j_{i-1}, a1_{C^j_{i-1}}]
\right) = \widetilde{a1_{L_j}}\prod_{i=1}^{n+1}\text{det}_{\Lambda^{\text{op}}}(\widetilde{a1_{C^j_{i-1}}})^{(-1)^i} $$ 
This is precisely the subgroup $\Xi$ of Definition \ref{defsha}, whence the result.  

\end{proof} 


\section{Existence of $n$-Cluster Tilting Objects in $\mcm R$}\label{sec4}

Naturally, the usefulness of Theorem \ref{thmg1} would be limited if the situations in which $\mcm R$ contained an $n$-cluster tilting object were sparse.  Fortunately for us, they are not.  Moreover, if $\mcm R$ admits an $n$-cluster tilting object $L$, we require that $\Lambda := \End_R(L)^{\text{op}}$ has finite global dimension.  At first glance, this condition might also seem limiting, but is in fact quite common, as seen in the following theorem. 

\begin{theorem}(\cite{i}, Theorem 3.12(a))\label{dn}

Suppose $\dim R = d$ and that $\mcm R$ contains an $n$-cluster tilting object $L$ with $d\leq n$.  Then $\Lambda$ has global dimension at most $n+1$.  

\end{theorem} 

The most well-studied situation in which $\mcm R$ admits an $n$-cluster tilting object is the following. 

\subsection{Finite Cohen-Macaulay Type}\label{ss14}

Recall that we say that $R$ has finite Cohen-Macaulay type (or \textit{finite type} for short) when $R$ has only finitely many indecomposable maximal Cohen-Macaulay modules.  Now the only $1$-cluster tilting subcategory of $\mcm R$ is $\mcm R$ itself.  Thus the existence of a $1$-cluster tilting object for $\mcm R$ is equivalent to $R$ having finite type.  In particular, when $R$ has finite type, $\mcm R$ has an additive generator $M$.  For practical and computational purposes, when $R$ has finite type, we will often work with the $R$-module $M = M_0\ds M_1\ds\cdots\ds M_t$, with $M_0 = R$ and $M_1,\ldots, M_t$ the pairwise non-isomorphic and non-free indecomposable maximal Cohen-Macaulay $R$-modules.  Moreover, (\cite{G}, Theorem 6) shows that $\End_R(M)^{\text{op}}$ has finite global dimension, hence Theorem \ref{thmg1} is applicable in this situation.  In fact, in this case, if the Auslander-Reiten homomorphism $\Upsilon:  \ints^{\ds t} \ra \ints^{\ds (t+1)}$ is injective, Theorem \ref{thmg1} is just (\cite{H}, Theorem 2.12), the result which inspired Theorem \ref{thmg1}.   

\subsubsection{ADE Singularities}\label{sss14} 

The most important examples of rings that have finite type are the simple surface singularities.  These are called the ADE singularities.  Let $S = k[[x, y, z_2, z_3,\ldots, z_d]]$  and assume $k$ is algebraically closed with characteristic different from $2, 3$ and $5$.  Set $R = S/fS$ with $f$ nonzero and $f\notin (x, y, z_2,\ldots, z_d)^2$.  The $f$ for which $R$ has finite type are exactly the following (\cite{GR}, Theorem 9.8) 
\begin{align*}
(A_n) &&   x^2 + y^{n+1} + z_2^2 + z_3^2 +\cdots + z_d^2 && (n\geq 1)
\\(D_n) && x^2y + y^{n-1} + z_2^2 + z_3^2 +\cdots + z_d^2 && (n\geq 4)
\\(E_6) & & x^3 + y^4 + z_2^2 + z_3^2 +\cdots + z_d^2
\\ (E_7) &&  x^3 + xy^3 + z_2^2 + z_3^2 +\cdots + z_d^2
\\ (E_8) && x^3 + y^5 + z_2^2 + z_3^2 +\cdots + z_d^2
\end{align*}
\subsection{Invariant Subrings}\label{ss24}

Let $k$ be a field and $S$ the ring $k[[x_1,\ldots, x_n]]$.  Suppose $G$ is a finite subgroup of $GL_n(k)$ that does not contain any nontrivial pseudo-reflections and with $|G|$ invertible in $k$.  Let $R$ be the invariant subring $k[[x_1,\ldots, x_n]]^G$ of $S$, where $G$ acts by a linear change of variables on $S$.  If $R$ is an isolated singularity, then the $R$-module $S$ is an $(n-1)$-cluster tilting object (see (\cite{i2}, 2.5)).  

\vs

The \textit{skew group ring} of $S$, denoted by $S\# G$, is given by $S\#G  = \bds_{\sigma\in G} S\cdot \sigma$,  with multiplication defined by $(s\cdot \sigma)(t\cdot \tau) = s\sigma(t)\cdot\sigma\tau$.  In this situation, $S\#G$ has global dimension equal to $n$ (\cite{GR}, Corollary 5.8) and there is an isomorphism $\End_R(S)\cong S\# G$ (\cite{GR}, Theorem 5.15).  In particular, Theorem \ref{thmg1} is applicable in this situation.

\subsection{Reduced Hypersurface Singularities}\label{ss34}

\subsubsection{Dimension One}\label{sss24}

Let $k$ be an algebraically closed field of characteristic not two and $S = k[[x, y]]$.  For $f\in (x,y)$, let $R = S/fS$ be a reduced hypersurface singularity.  Suppose $f$ has prime factorization and $f = f_1\cdots f_n$, $S_i = S/(f_1\cdots f_i)S$ and $L$ is the $R$-module $S_1\ds\cdots\ds S_n$.  If $f_i\notin (x, y)^2$ for all $i$, then \cite{hc} shows that $L$ is a $2$-cluster tilting object for $\mcm R$.  Moreover, Theorem \ref{dn} shows that $\End_R(L)^{\text{op}}$ has global dimension at most three.  Hence we can apply Theorem \ref{thmg1} in this situation.  Note, in particular, if $\lambda_1,\ldots, \lambda_n$ are distinct elements of $k$, then Theorem \ref{thmg1} is applicable to the ring $S/fS$ with $f = (x - \lambda_1y)\cdots (x-\lambda_ny)$.  

\subsubsection{Dimension Three}\label{sss34}

Keep notation as above, but set $S' = k[[x, y, u, v]]$ and $R' = S'/(f+uv)S'$.  Then $\mcm R'$ has a $2$-cluster tilting object if $f_i\notin (x,y)^2$ for all $i$ and it is given by $L := U_1\ds\cdots\ds U_n$ with $U_i = (u, f_1\cdots f_i)\subset R'$ (\cite{i}, Theorem 4.17).  Moreover, (\cite{i}, Theorem 4.17) also says $\End_{R'}(L)^{\text{op}}$ has finite global dimension, so Theorem \ref{thmg1} is applicable in this situation.   


\section{Abelianization of Automorphism Groups}\label{abgroup} 

Of course, the usefulness of Theorem \ref{thmg1} would be limited if one were unable to compute $\aut_R(L)_\ab$.  We make several computations, though each computation is tailored specifically to each ring and it seems difficult to find results that hold generally.  Our computations rely significantly upon the general framework laid out by \cite{H} and this work serves strongly as inspiration for our results.  The purpose of this section is to prove the following.  

\begin{proposition}\label{mainab}

Let $k$ be an algebraically closed field of characteristic not equal to two.  Then 

\vs 

(a)  if $R = k[x]/x^nk[x]$ and $M =  R \ds xR \ds\cdots \ds x^{n-1}R$, then $\aut_R(M)_\ab\cong (k^*)^{\oplus n}$.

\vs

(b)  if $k$ also has characteristic not equal to $3$ or $5$, $R = k[[t^2, t^{2n+1}]]$, $n\geq 0$ and $M = R \ds R_1\ds\cdots \ds R_n$, with $R_ i = k[[t^2, t^{2(n-i)+1}]$, then $\aut_R(M)_\ab \cong (k^*)^{\oplus n}\ds k[[t]]^*$.  

\vs 

(c) if $R = k[[s^2, st, t^2]]$, then $\aut_R(R\ds (s^2, st)R)_\ab\cong k^*\ds R^*$.  

\vs 

(d)  if $S = k[[x, y]]$, $f_1,\ldots, f_n\in (x, y)$ are irreducible such that 

\vs

   \hspace{.5 cm}(i)  $f = f_1\cdots f_n$, $R := S/fS$,  is an isolated singularity (i.e. $(f_i) \neq (f_j)$),
		
		\vs
		
		\hspace{.5 cm}(ii)  $f_i\notin (x, y)^2$ for all $i$,
	
	\vs
	
	\hspace{.5 cm}(iii)  $(f_i, f_{i+1}) = (x, y)$,  
	
	\vs 
	
	 $S_i = S/(f_1\cdots f_i)S$, and $L = S_1\ds\cdots\ds S_n$, then 
$$\aut_R(L)_\ab \cong (S/f_1S)^*\ds\cdots\ds (S/f_nS)^* = \overline{R}^*$$ 
Where $\overline{R} = S/f_1S\ds\cdots\ds S/f_nS$ is the integral closure of $R$ in its total quotient ring.  

\vs 
		
(e) if $k$ has characteristic zero, $S' = k[[x, y, u, v]]$,  $R' = S'/(f+uv)S'$, where $f = f_1\cdots f_n$ with $f_i\in k[[x, y]]$ satisfying the conditions in (d),  $U_i = (u, f_1\cdots f_i)$, and  $L = U_1\ds\cdots\ds U_n$, then 
$$\aut_{R'}(L)_\ab \cong  R'^* \ds k[[w, z]]^{*\ds (n-1)}$$
where $w, z$ are variables over $k$.  

\end{proposition} 

Of course, the purpose of Proposition \ref{mainab} is to combine it with Theorem \ref{thmg1} calculate explicit examples of $G_1(R)$ for several hypersurface singularities.  This will be done in Section \ref{examples}.  

\vs

We set up some useful notation.  Let $N_1,\ldots, N_s$ be $A$-modules and consider the $A$-module $N: = N_1\ds\cdots\ds N_s$.  We view the elements of $N$ as column vectors and the endomorphism ring of $N$ has a matrix-like structure:  For $f\in \End_A(N)$, we can write $f = (f_{ij})$ with $f_{ij}\in\hm_A(N_j, N_i)$ and composition with another endomorphism $g = (g_{ij})$ can be accomplished in the same manner one would multiply matrices with entries in $A$.  We write a $\text{diag}(\alpha_1,\ldots, \alpha_s)$ for the diagonal endomorphism of $N$ with $\alpha_i \in \End_A(N_i)$.  For $\alpha\in \aut_A(N_j)$, we denote by $d_j(\alpha)$ the automorphism of $N$ given by $\text{diag}(1_{N_1},\ldots, 1_{N_{j-1}}, \alpha, 1_{N_{j+1}},\ldots, 1_{N_s})$.  For $i\neq j$ and $\beta\in \hm_A(N_j, N_i)$, we denote by $e_{ij}(\beta)$ the automorphism of $N$ with diagonal entries $1_{N_1},\ldots, 1_{N_s}$ and $(i, j)$th entry given by $\beta$ and zeros elsewhere.  Before we begin, we discuss calculations that will be used often in the sequel.  

\begin{lemma} (\cite{H}, Lemma 9.2)\label{holmlemma} 

Let $A$ be a ring in which $2$ is invertible, $N_1,\ldots, N_s$ be $A$-modules and $N := N_1 \ds \cdots \ds N_s$.  If $i\neq j$ and $\alpha\in\hm_A(N_j, N_i)$, then $e_{ij}(\alpha)$ is a commutator in $\aut_A(N)$.   
\end{lemma} 
\begin{proof} 
Given $\beta, \gamma$ in $\aut_A(N)$, the commutator of $\beta$ and $\gamma$ is $[\beta, \gamma] = \beta\gamma\beta^{-1}\gamma^{-1}$.  It is not hard to see that $e_{ij}(\alpha) = [e_{ij}(\frac{\alpha}{2}), d_j(-1_{N_j})]$.  
\end{proof} 
\begin{lemma}\label{commutator} 
Let $(A, \mfn)$ be commutative and local such that $2$ is invertible in $A$.  Let $N_1,\ldots, N_s$ be $A$-modules and set $N = N_1 \ds \cdots \ds N_s$.  Let $a\in 1 + \mfn$, and consider the automorphism $d_i(a1_{N_i})d_{i+1}(a^{-1}1_{N_{i+1}})$ of $N$.  Suppose either 

\vs 

(a)  $N_i\supseteq N_{i+1}$ and $\mfn N_i\subseteq N_{i+1}$ or 

\vs 

(b) $N_i\subseteq N_{i+1}$ and $(1-a)N_{i+1}\subseteq N_i$ 

\vs 

then $d_i(a1_{N_i})d_{i+1}(a^{-1}1_{N_{i+1}})$ is in the commutator subgroup of $\aut_A(N)$.  
\end{lemma}  
\begin{proof} 
In the case of (a), Let $\iota_i:  N_{i+1}\ra N_i$ be inclusion.  Now note that $a^{-1}\in 1+\mfn$, so that we have the following decomposition of $d_i(a1_{N_i})d_{i+1}(a^{-1}1_{N_{i+1}})$:  
$$e_{i+1,i}((a^{-1} -1)1_{N_i})e_{i, i+1}(\iota_i)e_{i+1, i}((a-1)1_{N_i})e_{i, i+1}(-a^{-1}\iota_i)$$ 
We apply Lemma \ref{holmlemma} to see that $d_i(a1_{N_i})d_{i+1}(a^{-1}1_{N_{i+1}})$ is in the commutator subgroup of $\aut_A(N)$.  

\vs 

In the case of (b), notice that our hypothesis implies $(a^{-1}-1)N_{i+1}\subseteq N_i$.  Let $\iota_i:  N_i\ra N_{i+1}$ be the inclusion map.  We have the following decomposition of $d_i(a1_{N_i})d_{i+1}(a^{-1}1_{N_{i+1}})$:  
$$e_{i, i+1}((a^{-1} - 1)1_{N_{i+1}})e_{i+1, i}(\iota_i)e_{i, i+1}((a - 1)1_{N_{i+1}})e_{i+1, i}(-a^{-1}\iota_i)$$ 
and once again, we apply Lemma \ref{holmlemma} to see that $d_i(a1_{N_i})d_{i+1}(a^{-1}1_{N_{i+1}})$ is in the commutator subgroup of $\aut_A(N)$.  
\end{proof}  
\subsection{Truncated Polynomial Rings in One Variable}\label{ss15}
Our aim here is to prove (a) of Proposition \ref{mainab}.  That is $k$ is algebraically closed and has characteristic not two, $R = k[x]/x^nk[x]$, $\mfm$ is its maximal ideal $xR$, then with $M = R \ds \mfm \ds\cdots \ds \mfm^{n-1}$, we have $\aut_R(M)_\ab \cong (k^*)^{\ds n}$.  
\begin{proof} 
Denote by $R_j$ the ring $k[x]/x^jk[x]$ for $1\leq j\leq n$.  Note that $R_{j-1}\subset R_j$ and $R = R_n$.  Let $\mfm$ denote the maximal ideal $xR$ of $R$.  Then $\End_R(\mfm^i)$ is isomorphic to the local ring $R_{n-i}$.  Let $M$ be the $R$-module $ R\ds \mfm\ds\cdots\ds \mfm^{n-1}$.  We set $E = \End_R(M)$ and seek to show $E^*_\ab \cong (k^*)^{\ds n}$.  

\vs

For $n = 1$, this is clear.  For $n = 2$, we have $\End_R(\mfm)  = k$, so that $E^*_\ab \cong (k^*)^{\ds 2}$ by (\cite{H}, Proposition  9.6).  

\vs 

Suppose now $n\geq 3$.  We first show that there is a surjection $E^*_{\text{ab}}\ra (k^*)^{\oplus n}$ such that the kernel consists of diagonal matrices $\alpha = (\alpha_{ii})$ with $\alpha_{ii}\in \aut_R(\mfm^{i-1}) = R_{n-i+1}^*$.   By (\cite{H}, Proposition 9.4), $(\alpha_{ij})\in E$ is invertible if and only if $\alpha_{ii}$ is invertible for all $i$.  In particular, this gives that every two-sided maximal ideal of $E$ is of the form $ \mfn_i := \stl (\alpha_{ij}) : \alpha_{ii}\in J(\End_R(\mfm^{i-1}))\str$.  Hence the Chinese Remainder Theorem gives $E/J(E) \cong E/\mfn_1\times\cdots\times E/\mfn_n = k\times\cdots\times k$.  In particular, there is an induced surjection $\vp:  E^*_\ab \twoheadrightarrow (k^*)^{\ds n}$.  We appeal to (\cite{H}, Corollary 9.5) to see every element of $E^*_{\text{ab}}$ can be represented by a diagonal automorphism.  Moreover, it is clear elements in the kernel $\vp$ are given by $(\alpha_{ii})$ such that $\alpha_{ii}$ is multiplication by an element in $1 +J(\End_R(\mfm^{i-1})) = 1 + xR_{n-i+1}$ for all $i$.  

\vs 

We now demonstrate the injectivity of $\vp$.  Let $\alpha\in E^*_{\ab}$ such that $\vp(\alpha)$ is trivial.  By the above, we can write $\alpha = (\alpha_{ii})$ such that $\alpha_{ii}$ is multiplication by an element of $1 + xR_{n-i+1}^* $.  Now every endomorphism on $\mfm^{n-1}$ is given by an element of $1 + xR_1 = \left\{1\right\}$, so that we can write $\alpha = d_1(\alpha_{11})\cdots d_{n-1}(\alpha_{n-1,n-1})$.  It suffices to show each $d_i(\alpha_{ii})$ is in the commutator subgroup of $E^*$.  We do this below.  

\vs

We show by decreasing induction on $i$ that $d_i(\beta)$ can be written as a product of commutators, where $\beta$ is given by multiplication by an element of $1 + xR_{n-i+1}$.  For $i = n-1$, write $\beta= r1_{\mathfrak m^{n-2}}$, where $r\in 1 +xR_2$.  Notice that $r^{-1}\in 1 +xR_2$ as well, hence multiplication by $r^{-1}$ restricts to the identity on $\mathfrak m^{n-1}$.  This gives 

$$d_{n-1}(\beta) = d_{n-1}(r1_{\mathfrak m^{n-2}}) = d_{n-1}(r1_{\mathfrak m^{n-2}})d_n(r^{-1}1_{\mathfrak m^{n-1}})$$ 

By Lemma \ref{commutator}, $d_{n-1}(r1_{\mathfrak m^{n-2}})d_n(r^{-1}1_{\mathfrak m^{n-1}})$ is in the commutator subgroup of $E^*$, hence so is $d_{n-1}(\beta)$.  Suppose now $i < n-1$ and $\beta\in\aut_R(\mfm^{i-1})$ is given by multiplication on $\mfm^{i-1}$ by an element of $1 +xR_{n-i+1}$.  We have 

$$d_i(\beta) = d_i(\beta)d_{i+1}(\beta^{-1}|_{\mfm^i})d_{i+1}(\beta |_{\mfm^i})$$ 

By the induction hypothesis, $d_{i+1}(\beta |_{\mfm^i})$ is in the commutator subgroup of $E^*$.  By Lemma \ref{commutator}, $d_i(\beta)d_{i+1}(\beta^{-1}|_{\mfm^i})$ is in the commutator subgroup of $E^*$, hence so is $d_i(\beta)$.  This completes the induction step and gives that $E^*_\ab \cong (k^*)^{\ds n}$.  
\end{proof} 
\subsection{Singularty of Type $A_{2n}$ in Dimension One}\label{ss25}
Our aim here is to prove (b) of Proposition \ref{mainab}.  Thus, $k$ is an algebraically closed field of characteristic not equal to  $2, 3$ or $5$ and $R$ the ring $k[[t^2, t^{2n+1}]]$.  Set $R = R_0$ and let $M$ be the $R$-module $R_0\ds R_1\ds\cdots \ds R_n$, where $R_i = k[[t^2, t^{2(n-i)+1}]]$ for $i = 0,\ldots, n$.  Then we want to show $\aut_R(M)_\ab \cong (k^*)^{\ds n} \ds k[[t]]^*$.  Before we begin, we prove the following.  

\begin{lemma}\label{homan} 
Let $0\leq i, j\leq n$.  If 

\vs 

(a)  $i\leq j$, then $\hm_R(R_i, R_j) = R_j$.  

\vs 

(b) $i > j$, then $\hm_R(R_i, R_j)$ can be viewed as a subset of $R$.  In particular, it is contained in $R_n = k[[t]]$.   

\vs 

As a consequence of the above, we can view $E := \End_R(M)$ as a subring of $\mathbb M_{n+1}(R_n) = \mathbb M_{n+1}(k[[t]])$.  

\end{lemma}  

\begin{proof} 
(a)  We claim 
$$\hm_R(R_i, R_j) = \hm_{R_i}(R_i, R_j)$$ 
Indeed, since $R\subseteq R_i$, there is a natural inclusion $\hm_{R_i}(R_i, R_j)\subseteq \hm_R(R_i, R_j)$.  We demonstrate the reverse inclusion.  Let $s\in R_i$, $f\in R_i$ and $\varphi\in \hm_R(R_i, R_j)$.  It is not hard to see that there is a nonzero $r\in R$ such that $rs\in R$ (for example, by noting that $t^{2n+1}\in R$, $t^{2n+1}k[[t]]\subseteq R$, and $R_i\subseteq k[[t]]$).  We have 
$$r\vp(sf) = \vp(rsf) = rs \vp(f)$$ 
and $r$ is nonzero, so that $\vp(sf) = s\vp(f)$.  This proves the claim.  Thus, we have 
$$\hm_R(R_i, R_j) = \hm_{R_i}(R_i, R_j)$$ 
and the latter is naturally isomorphic to $R_j$.  

\vs 

(b)  By (\cite{int}, Lemma 2.4.3), there is an isomorphism of $R$-modules:  
$$\hm_R(R_i, R_j) \cong (R_j :_R R_i)$$ 
Where $(R_j :_R R_i)$ is the ideal of $R$ consisting of $f\in R$ such that $fR_i\subseteq R_j$.  

\vs 

Utilizing (a) and (b), we see that $E$ can be viewed as the subring of $\mathbb M_{n+1}(R_n) = \mathbb M_{n+1}(k[[t]])$ given by 
$$\begin{pmatrix}
R_0 & R_1 & R_2 & R_3 &\cdots & R_n \\
(R_1 :_R R_0) & R_1 & R_2 & R_3 &\cdots & R_n \\
(R_2 :_R R_0) & (R_2 :_R R_1) & R_2 & R_3 &\cdots & R_n  \\
\vdots &\vdots&\vdots &\vdots &\vdots&\vdots \\
(R_n :_R R_0) & (R_n :_R R_1) & (R_n :_R R_2) & (R_n :_R R_3) &\cdots & R_n\\
\end{pmatrix}$$
\end{proof} 

We now proceed with the proof of (b) of Proposition \ref{mainab}.  

\begin{proof}  
Note the $R_i$ are finitely generated $R$-modules; each $R_i$ is local with maximal ideal $\mfm_i = (t^2, t^{2(n-i)+1})R_i$; we have inclusions $R_i\subseteq R_{i+1}$ and $\mfm_i\subseteq \mfm_{i+1}$; and each $R_i$ has $k$ as a residue field.  

\vs

This is clear for $n = 0$.  For $n = 1$, this is just (\cite{H}, Proposition  9.6), since $k[[t]]\cong (t^2, t^3)k[[t^2, t^3]]$ as $k[[t^2, t^3]]$-modules.  

\vs 

Suppose now $n\geq 2$.  Our goal is to construct a map from $E^*$ to the abelian group $(k^*)^{\ds n}\ds k[[t]]^*$, so that we obtain an induced map $E^*_\ab \ra (k^*)^{\ds n} \ds k[[t]]^*$ that we will later show is an isomorphism.  

\vs 

First we construct a map from $E^*$ to $(k^*)^{\ds n}$.  The proof that there is group homomorphism from $E^* \ra (k^*)^{\ds n}$ works in exactly the same manner as as it did in the proof of (a) of Proposition \ref{mainab}.  Noting of course that with $\mfn_i = \stl (\alpha_{ij}) : \alpha_{ii}\in J(\End_R(R_{i-1}))\str$, (a) of Lemma \ref{homan} gives that $E/\mfn_i = \End_R(R_{i-1})/J(\End_R(R_{i-1}))  = R_{i-1}/\mfm_{i-1} = k$.  Thus we obtain an induced map $E^*_\ab\ra (k^*)^{\ds n}$. 
 
\vs 

As Lemma \ref{homan} allows us to view $E$ as a subring of $\mathbb M_{n+1}(R_n) = \mathbb M_{n+1}(k[[t]])$, $E^*$ is naturally a subset of $GL_{n+1}(k[[t]])$, the group of invertible $(n+1)\times (n+1)$ matrices over $k[[t]]$.  By taking the determinant, we obtain a map from $E^* \ra k[[t]]^*$.  Now $k[[t]]^*$ is abelian, hence this induces a group homomorphism $E^*_\ab\ra k[[t]]^*$.  

\vs

Regarding $E$ as a matrix subring of $\mathbb M_{n+1}(k[[t]])$, we combine our preceding work to see there is a group homomorphism $\Phi: E^*_\ab\ra (k^*)^{\oplus n}\ds R_n^*$ such that the image of $\alpha = (\alpha_{ij})\in E^*_\ab$ under $\Phi$ is 
$$(\alpha_{11} + \mfm_0, \ldots, \alpha_{nn}+\mfm_{n-1}, \text{det}(\alpha))$$ 

We note $\Phi$ is surjective:  For $(a_1,\ldots a_n, f)\in (k^*)^{\oplus n}\ds R_n^*$,  $(a_1,\ldots a_n, f)$ is the image under $\Phi$ of
$$\text{diag}(a_1, a_2,\ldots, a_n, (a_1a_2\cdots a_n)^{-1}f)$$ 
To see that $\Phi$ is injective, let $\alpha\in E^*_\ab$ such that $\Phi(\alpha)$ is trivial.  By (\cite{H}, Corollary 9.5), we may assume that $\alpha\in E^*_\ab$ is diagonal.  Write $\alpha = \text{diag}(f_0, f_1,\ldots, f_{n-1}, f_n)$, with $f_{i-1}\in (R_{i-1})^*$ by Lemma \ref{homan}.  Since $\Phi(\alpha)$ is trivial, $f_{i-1}\in 1+\mfm_{i-1}$ for $i=1,\ldots, n$ and $f_0f_1\cdots f_n = 1$ in $R_n^* = k[[t]]^*$.  Hence for $i = 1,\ldots, n$, $\alpha$ is the product of the diagonal automorphisms $\beta_i = d_i(f_{i-1})d_{n+1}(f_{i-1}^{-1})$.  Consider the automorphisms  $\gamma_i = d_i(f_{i-1})d_{i+1}(f_{i-1}^{-1})$ and note that $\beta_i = \gamma_i\cdots \gamma_n$.  To see that $\gamma_i$ is in the commutator subgroup, note that $f_{i-1}^{-1}$ is in $1 + \mfm_{i-1}$, hence multiplication by $f_{i-1}-1$ maps $R_i$ into $R_{i-1}$.  Indeed, multiplication by $\mfm_{i-1}$ on $\mfm_i$ takes $\mfm_i$ into $\mfm_{i-1}$.  Moreover, any unit in $R_i$ is a power series with nonzero constant term, hence multiplication on $R_i$ by an element in $\mfm_{i-1}$ takes $R_i$ into $R_{i-1}$.  Thus the hypotheses of Lemma \ref{commutator} are satisfied, so that each $\gamma_i$ is in the commutator subgroup of $E^*$, hence so is each $\beta_i$, and ultimately so is $\alpha$.  Thus $\Phi$ is injective, hence an isomorphism.  
\end{proof} 

\subsection{ Generalities for Invariant Subrings}\label{ss35}

Let $k$ be a field.  Recall from Section \ref{sec4} that $S$ is the ring $k[[x_1,\ldots, x_n]]$, $G$ is a finite subgroup of $GL_n(k)$ that does not contain any nontrivial pseudo-reflections with $|G|$ invertible in $k$ and $R$ is the invariant subring $S^G$ of $S$ (where $G$ acts by a linear change of variables on $S$).  Then if $R$ is an isolated singularity, the $R$-module $S$ is an $(n-1)$-cluster tilting object in $\mcm R$.   

\vs 

We need the following lemmas for the proof of (c) of Proposition \ref{mainab}.  

\begin{lemma}\label{norm} 

Let $A$ be a local Cohen-Macaulay integral domain of dimension $d > 1$ such that $A$ is an isolated singularity.  Then $A$ is normal and $\hm_A(I, I) \cong A$ for any ideal $I$ of height one.  

\end{lemma} 

\begin{proof} 

Clearly $A$ satisfies Serre's criterion for normality.  For the second part, choose $x\in I$ to be nonzero.  Then (\cite{int}, Lemma 2.4.3) shows that $\hm_A(I, I)$ can be identified with the $A$-submodule $\frac{1}{x}(xI :_A I)$ of the quotient field of $A$.  Now $(Ix :_A I)$ is nonzero and contained in $I$, so must have height one.  If $I$ is principal, it is clear that $(xI :_A I) = xA$.  However, as $(xI :_A I)$ has height one and $A$ is an isolated singularity, $A_\mfp$ is a discrete valuation ring for every associated prime $\mfp$ of $(xI :_A I)$, hence $(xIA_\mfp :_{A_\mfp} IA_\mfp) = xA_\mfp$.  Thus $(xI :_A I) = xA$ and $\hm_A(I, I)\cong A$.  

\end{proof} 

\begin{lemma}(\cite{dfi1}, Lemma 5.4)\label{grade}

Let $A$ be a commutative Noetherian ring.  Then for any ideal $I$ and module $M$ such that $\text{grade}(I, M)\geq 2$, we have $\hm_A(I, M) \cong \hm_A(A, M) \cong M$.  

\end{lemma}

\subsection{Singularity of Type $A_1$ in Dimension Two}\label{sss15}

Our aim here is to prove (c) of Proposition \ref{mainab}.  Thus $R$ is the $A_1$ singularity $k[[s^2, st, t^2]]$ in dimension two with $\text{char}(k)\neq 2$.  If $I = (s^2, st)R$, then $\aut_R(R\ds I)_\ab\cong k^*\ds R^*$.  

\begin{proof} 

By (\cite{GR}, Example 5.25), the indecomposable maximal Cohen-Macaulay modules of $M$ are $R$ and $I$.  That is, $R$ has finite type.  Thus by (\cite{Y}, Theorem 4.22), $R$ is an isolated singularity.  Moreover, since $R$ is of finite type, $R\ds I$ is an additive generator for $\mcm R$, so that $\End_R(R\ds I)^{\text{op}}$ has finite global dimension by (\cite{G}, Theorem 6).  Now $I$ has height one, so that $\hm_R(I, I)\cong R$ by Lemma \ref{norm}.  Moreover, as $I$ is maximal Cohen-Macaulay, we have $\hm_R(I,R) \cong R$ by Lemma \ref{grade}.  Thus $\End_R(R\ds I)$ is isomorphic to the subring $\left(\begin{array}{ c c }
     R& R\\
   I & R
  \end{array} \right)$ of $\mathbb M_2(R)$.  By (\cite{pg}, Corollary 2.8), there is an isomorphism 
$$\aut_R(R\ds I)_\ab \cong K_1^C(R) \ds K_1^C(R/I) = R^* \ds k[[t^2]]^*$$ 
Thus if $\mfm$ denotes the maximal ideal of $R$, we have 
\begin{eqnarray}
R^*\ds k[[t^2]]^* & \cong &  k^* \ds 1+\mfm \ds k[[t^2]]^*  \nonumber \\
   &\cong & k^* \ds k[[t^2]][[s^2, st, t^2]]^*  \nonumber \\
   & = & k^* \ds R^* \nonumber
\end{eqnarray}
\end{proof} 

\subsection{Generalities for Reduced Hypersurface Singularities}\label{ss45}

Before we prove parts (d) and (e) of Proposition \ref{mainab}, we discuss another route for computing the group $\aut_R(L)_\ab$ that we plan to utilize for the proof.  We begin with another aside on noncommuatative algebra.  A ring $A$ with Jacobson radical $J(A)$ is said to be \textit{semiperfect} if $A$ is semilocal and idempotents of $A/J(A)$ lift to idempotents of $A$.  We assume that $\mcm R$ contains an $n$-cluster tilting object $L$ of the form $L_0\ds L_1\ds\cdots\ds L_t$ and $L_0, L_1,\ldots, L_t$ are pairwise non-isomorphic and indecomposable.  As $\End_R(L_i)$ is local for all $i$, it is the case that $\Lambda = \End_R(L)^{\text{op}}$ is semiperfect by (\cite{Lam}, Theorem 23.8) (noting that $\Lambda$ is semiperfect if and only if $\Lambda^{\text{op}}$ is semiperfect).  In particular, if $R$ is a $k$-algebra, the characteristic of $k$ is not two, then by (\cite{L}, Theorem  2), there is an isomorphism 
$$K^C_1(\Lambda) \cong \Lambda^*_\ab = \aut_R(L)_\ab$$ 
Since $\Lambda$ is semiperfect, with the above isomorphism, we can utilize (\cite{pg}, Theorem 2.2) to obtain an isomorphism 
$$\aut_R(L)_\ab \cong K^C_1(\Lambda)\cong \left ( \bds_{i=0}^t \aut_R(L_i)\right)\textbf{\bigg/}HC$$
Where $C$ is the subgroup of $\bds_{i=0}^t \aut_R(L_i)$ generated by all elements of the form 
$$(1+\alpha_i\beta_i)(1+ \beta_i\alpha_i)^{-1}$$
with $\alpha_i, \beta_i\in \End_R(L_{i-1})$ such that $1 + \alpha_i \beta_i\in \aut_R(L_{i-1})$, and $H$ is the subgroup generated by all elements of the form 
$$(1+\alpha_{ij}\alpha_{ji})(1 + \alpha_{ji}\alpha_{ij})^{-1}$$
with $\alpha_{ij}\in \hm_R(L_{i-1}, L_{j-1})$, $i\neq j$, and $1+\alpha_{ij}\alpha_{ji}\in \aut_R(L_{i-1})$.  

\vs 

However each $\alpha_{ij}\alpha_{ji}$ is never an automorphism when $i\neq j$, since otherwise $L_{i-1}$ would be a direct summand of $L_{j-1}$ (see (\cite{H}, Lemma 9.3).  Since each of the rings $\End_R(L_{i-1})$ are local, this implies that $1+\alpha_{ij}\alpha_{ji}\in \aut_R(L_{i-1})$ \textit{for all} $i\neq j$.  

\vs 

We now continue with the proof of (d) of Proposition \ref{mainab}.  

\subsection{Reduced Hypersurface Singularities in Dimension One}\label{sss35}

Our aim here is to prove (d) of Proposition \ref{mainab}.  Here, $k$ is an algebraically closed field of characteristic not two and $S = k[[x, y]]$, $R$ is the ring $S/fS$ with $f\in (x, y)$ is such that in its prime factorization, $f = f_1\cdots f_n$ we have $(f_i)\neq (f_j)$ for $i\neq j$, $f_i\notin (x, y)^2$, $(f_i, f_{i+1}) = (x, y)$.  Then if $S_i = S/(f_1\cdots f_i)S$ and $L := S_1\ds\cdots\ds S_n$, we have $\aut_R(L)_\ab  \cong \ol{R}^*$, where $\ol{R}$ is the integral closure of $R$ in its total quotient ring.  We first prove a useful lemma.  

\begin{lemma}\label{dim1}
With notation as above, we have 
\begin{equation*}\hm_R(S_j, S_i) \cong \left\{\begin{array}{lr}
(f_{j+1}\cdots f_i)/(f_1\cdots f_i)&  j < i\\
S_i & i\leq j\end{array}\right.
\end{equation*}
\end{lemma}  

\begin{proof} 
The isomorphisms 
$$\hm_R(S_j, S_i) \cong \hm_R(R/(f_1\cdots f_j), R/(f_1\cdots f_i)) \cong (0 :_{R/(f_1\cdots f_i)} (f_1\cdots f_j)) $$   
make the statement clear.  
\end{proof} 

We now proceed with the proof of (d) of Proposition \ref{mainab}.  
\begin{proof} 

Now by (\cite{hc}, 4.7), $L$ is a $2$-cluster-tilting object for $\mcm R$.  As $\Lambda := \End_R(L)^{\text{op}}$ has finite global dimension by Theorem \ref{dn}, the remarks of Subsection \ref{ss45} give 
$$\aut_R(L)_\ab \cong \left(\bds_{i=1}^n \aut_R(S_i)\right)/HC$$ 
Where $C$ is the subgroup of $\bds_{i=1}^n \aut_R(S_i)$ generated by all elements of the form $(1+\alpha_i\beta_i)(1+\beta_i\alpha_i)^{-1}$ such that $ \alpha_i,\beta_i\in \End_R(S_i)$ and $1+\alpha_i\beta_i\in \aut_R(S_i)$.  By Lemma \ref{dim1}, $\End_R(S_i) = S_i$, so that $C$ is trivial and $\aut_R(L)_\ab \cong \left(S_1^*\ds \cdots \ds S_n^*\right)/H$.  We now describe the subgroup $H$.  Again by the remarks in subsection \ref{ss45}, $H$ is the subgroup generated by all elements of the form 
$$(1+\alpha_{ij}\alpha_{ji})(1+\alpha_{ji}\alpha_{ij})^{-1}$$
where $\alpha_{ij}\in\hm_R(S_j, S_i)$,  $\alpha_{ji}\in\hm_R(S_i, S_j))$, and $i\neq j$.  In fact, we can consider the subgroup generated by such elements where $i < j$.  We note $\alpha_{ij}\alpha_{ji}\in\End_R(S_i) = S_i$ and $\alpha_{ji}\alpha_{ij}\in\End_R(S_j) = S_j$.  Utilizing Lemma \ref{dim1}, we can give a more concise description of $H$ as follows (note $i < j$).  The subgroup $H$ is generated by the elements $h_{ij}(s)$, which we now describe:   

\vs 

(i)  the $i$th entry of $h_{ij}(s)$ is the image of an element $s\in 1+ (f_{i+1}\cdots f_j)\subset S$ in the unit group $S_i^*$; 

\vs

(ii) the $j$th entry of $h_{ij}(s)$ is the image of $s^{-1}$, with $s$ from (i) in the unit group $S_j^*$; 

\vs

(iii)  $h_{ij}(s)$ is trivial elsewhere.  

\vs

Let $H_{i, j}$ be the subgroup of $H$ generated by the $h_{ij}(s)$, with $s$ defined above.  We have $H = \ds_{i < j} H_{i, j}$.  By projecting onto the $j$th coordinate, it is easy to see $H_{i, j}$ is isomorphic to the subgroup $1 + (f_{i+1}\cdots f_j)$ of $S_j^*$.  For $1\leq i < n$, we call the subgroup $H_{i, i+1}\ds\cdots \ds H_{i, n}$ of $H$ the \textit{$i$th layer} of $H$.  It is easy to see that $S_1^*\ds \cdots \ds S_n^*$ modulo the direct sum of the first $m$ layers of $H$ is 
$$\bds_{u = 1}^{m+1} (S/f_u S)^* \ds \bds_{v = m+2}^n (S/(f_{m+1}\cdots f_v)S)^*$$ 
As $H$ is the direct sum of its $n-1$ layers of $H$, we see that $S_1^*\ds \cdots \ds S_n^*$ modulo $H$ is just 
$$(S/f_1S)^*\ds\cdots\ds (S/f_nS)^*$$
And this is just $\ol{R}^*$.  

\end{proof} 

\subsection{Reduced Hypersurface Singularities in Dimension Three}\label{sss45}

Our aim here is to prove (e) of Proposition \ref{mainab}.  Keep notation as in Subsection \ref{sss35} with the exception that we now require $k$ be an algebraically closed field of characteristic zero.  Set $S' = k[[x, y, u, v]]$ and $R' = S'/(f+uv)S'$.  Then a $2$-cluster tilting object for $\mcm R$ is given by $L := U_1\ds\cdots\ds U_n$, with $U_i = (u, f_1\cdots f_i)\subset R'$ (see Section \ref{sec4}).  Then we aim to show $\aut_R(L)_\ab \cong R'^* \ds k[[w, z]]^{*\ds (n-1)}$.  In order to understand $\aut_R(L)_\ab$, we first need to understand the structure of the modules $\hm_{R'}(U_i, U_j)$, so that we are able to use the remarks of Subsection \ref{ss45} to compute $\aut_R(L)_\ab$.  This is the first step we make below.  

\begin{proposition}\label{dim3}

Let $R'$ and $U_i$ be as above.  Then 
\begin{displaymath}
   \hm_{R'}(U_i, U_j) \cong \left\{
     \begin{array}{lr}
       U_j & j < i \ \\
       R' & i \leq j 
     \end{array}
   \right.
\end{displaymath} 
\end{proposition} 

\begin{proof} 
Now $R'$ is Gorenstein of dimension three and an isolated singularity.  Since $U_i$ is an ideal of $R'$ of height one, we may apply Lemma \ref{norm} to see that $\hm_{R'}(U_i, U_i)\cong R'$ for all $i$.  If $i\neq j$, (\cite{int}, Lemma 2.4.3) says we may identify $\hm_{R'}(U_i, U_j)$ with the the $R'$-submodule $\frac{1}{u}(uU_j :_{R'} U_i)$ of the quotient field of $R'$.  Now $(uU_j :_{R'} U_i)$ is nonzero and $(uU_j :_{R'} U_i)\subset U_j$, hence $(uU_j :_{R'} U_i)$ has height one.  Let $\mfp$ be a minimal prime of $(uU_i :_{R'} U_j)$.  As $R'$ is an isolated singularity, $R'_\mfp$ is a discrete valuation ring.  Write $R'_\mfp = A$ and let $\mu$ be a generator for the maximal ideal of $A$.  Suppose $u$ maps to $c\mu^a$, with $a > 0$ and $c\in A^*$.  Write $(U_i)_\mfp = \mu^{n_i}A$ and $(U_j)_\mfp = \mu^{n_j}A$, with $n_j, n_i$ nonnegative integers.  Then 
$$(uU_j :_{R'} U_i)_\mfp = (\mu^{a+n_j} :_A \mu^{n_i})$$ 
If $i \leq j$, then $U_j \subseteq U_i$, hence $n_j\geq n_i$.  We have 
$$(\mu^{a+n_j} :_A \mu^{n_i}) = \mu^{a + n_j -n_i}A\subset \mu^aA$$
Thus $(uU_j :_{R'} U_i)_\mfp = (uR')_\mfp$.  In this case, $(uU_i :_{R'} U_j) = uR'$, so that $\hm_{R'}(U_i, U_j) \cong R'$.  

\vs 

Now if $j < i$, then $U_i\subset U_j$ and $n_i \geq n_j$.  Notice $u\in U_i$, so that $a\geq n_i$.  We have 
$$(\mu^{a+n_j} :_A \mu^{n_i}) = \mu^{a + n_j -n_i} = \mu^{a-n_i}(U_j)_\mfp$$
And $\mu^{a-n_i}A = (\mu^a :_A \mu^{n_i}) = (u:_{R'} U_i)_\mfp$.  We have $(u:_{R'} U_i) = (u:_{R'} f_1\cdots f_i)$.  Thus $(uU_j :_{R'} U_i) = (u:_{R'} f_1\cdots f_i)U_j$.  When $i = n$, $(f_1\cdots f_n)R' = (uv)R'$, hence $(u :_{R'} f_1\cdots f_n) = R'$.  This gives $U_j = (uU_j :_{R'} U_n)$, hence there is an isomorphism of $R'$-modules $\hm_{R'}(U_n, U_j)\cong \frac{1}{u}U_j\cong U_j$.  

\vs 

To analyze the ideal $(u :_{R'} f_1\cdots f_i)$ for $i < n$, note that $f_{i+1}\cdots f_n\in (u :_{R'} f_1\cdots f_i)$ and that the ideals $(u, f_{i+1})R',\ldots, (u, f_n)R'$ are prime.  In particular, the minimal primes of $(u :_{R'} f_1\cdots f_i)$ are $(u, f_{i+1})R',\ldots, (u, f_n)R'$.  Let $\mfq$ denote the prime ideal $(u, f_s)R'$, with $i+1\leq s\leq n$.  Then $(f_1\cdots f_i)R'_\mfq = R'_\mfq$, as $f_1,\ldots, f_i\notin \mfq$ and hence $(u :_{R'} f_1\cdots f_i)_\mfq = (uR')_\mfq$.  Thus $(u :_{R'} f_1\cdots f_i) = uR'$, so that $(u:_{R'} f_1\cdots f_i)U_j = uU_j$, and hence $\hm_{R'}(U_i, U_j)\cong U_j$.  This gives the result.  
\end{proof} 

We now proceed with the proof of (e) of Proposition \ref{mainab}.  

\begin{proof} 
By  (\cite{i}, Theorem 4.17), $\End_{R'}(L)^{\text{op}}$ has finite global dimension, so the remarks of Subsection \ref{ss45} are applicable.  Thus there is an isomorphism:  
$$\aut_{R'}(L)_\ab \cong \left(\bds_{i=1}^n \aut_{R'}(U_i)\right )/HC$$ 
Where $C$ is the subgroup of $\bds_{i=1}^n \aut_{R'}(U_i)$ generated by $(1+\alpha_i\beta_i)(1+\beta_i\alpha_i)^{-1}$ such that $ \alpha_i,\beta_i\in \End_{R'}(U_i)$ and $1+\alpha_i\beta_i\in \aut_{R'}(U_i)$.  By Proposition \ref{dim3}, $\End_{R'}(U_i) = R'$, so that $C$ is trivial, hence $\aut_{R'}(L)_\ab \cong (R'^*)^{\ds n}/H$.  We now describe the subgroup $H$.  Again by the remarks in subsection \ref{ss45}, $H$ is the subgroup generated by all elements of the form 
$$(1+\alpha_{ij}\alpha_{ji})(1+\alpha_{ji}\alpha_{ij})^{-1}$$ 
where $\alpha_{ij}\in\hm_{R'}(U_j, U_i)$,  $\alpha_{ji}\in\hm_{R'}(U_i, U_j))$, and $i\neq j$.  In fact, we can consider the subgroup generated by such elements where $i < j$.  We now give a more concise description of $H$.  Utilizing Proposition \ref{dim3}, $H$ is the subgroup of $(R'^*)^{\ds n}$ generated by the elements $h_{ij}(g)$ with $i < j$ and $g\in U_i$ such that that: 

\vs 

(i)  the $i$th entry of $h_{ij}(g)$ is $1+g$; 

\vs

(ii)  the $j$th entry of $h_{ij}(g)$ is $(1+g)^{-1}$; 

\vs

(iii)  $h_{ij}(g)$ is trivial elsewhere.  

\vs 

For fixed $i$ and $j$, let $H_{i, j}$ be the subgroup generated by the elements $h_{ij}(g)$.  Thus $H = \ds_{i < j} H_{i, j}$ and $H_{i, j} \cong 1 + U_i \subset R'^*$.  For $i < n$, we call the subgroup $H_{i, i+1} \ds H_{i, i+2}\ds\cdots\ds H_{i, n}$ the \textit{ith layer} of $H$.  As $U_n\subset U_n\subset\cdots\subset U_1$, it is easy to see that $(R'^*)^{\ds n}$ modulo the direct sum of layers $n-1, n-2,\ldots, n-i$ is isomorphic to 
$$(R'^*)^{\ds (n-i)} \ds (R'/U_{n-i})^{* \ds i}$$
Now the direct sum of layers $n-1, n-2,\ldots, 1$ is just $H$, so that we see 
$$\aut_{R'}(L)_\ab \cong R'^* \ds (R'/U_1)^{* \ds (n-1)}$$ 
Moreover, since $U_1 = (u, f_1)$ and $f_1\in (x,y)\backslash (x, y)^2\subseteq k[[x, y]]$, we see $R'/U_1\cong k[[w, z]]$, for variables $w, z$ over $k$.  Thus 
$$\aut_{R'}(L)_\ab \cong R'^* \ds k[[w, z]]^{* \ds (n-1)}$$

\end{proof} 


\section{Computing $G_1(R)$}\label{examples}
The aim of this section is to utilize Theorem \ref{thmg1} to explicitly calculate $G_1(R)$ for several hypersurface singularities.  Our results are the following:  

\begin{example}\label{ex1} 
Let $k$ be an algebraically closed field of characteristic not two.  If $n\geq 1$ and $R = k[x]/x^nk[x]$, then $G_1(R)\cong k^*$.  

\end{example}  

\begin{remark}
We note that Example \ref{ex1} follows immediately from Quillen's D\'{e}vissage Theorem (\cite{Q}, \S 5 Theorem 4), but we find the calculation illustrative of our methods as well as allowing us to generalize (\cite{H}, Example 10.2).  

\end{remark} 

\begin{example}\label{ex2} 
Let $k$ be an algebraically closed field of characteristic not two, three or five.  If $R$ is the finite-type singularity $k[[t^2, t^{2n+1}]]$ for $n\geq 0$, then $G_1(R)\cong\overline{R}^* = k[[t]]^*$;    

\end{example}  

\begin{example}\label{ex3} 
Let $k$ be an algebraically closed field of characteristic not two.  If $S = k[[x, y]]$ let $f_1,\ldots, f_n\in (x, y)$ be irreducible and $f = f_1\cdots f_n$ be such that 

\vs

    \hspace{.5 cm}(i)  $R := S/fS$ is an isolated singularity (ie. $(f_i) \neq (f_j)$) 
		
		\vs
		
		\hspace{.5 cm}(ii)  $f_i\notin (x, y)^2$ for all $i$.  
	
	\vs
	
	\hspace{.5 cm}(iii)  $(f_i, f_{i+1}) = (x, y)$.  
	
	\vs
	
Then $G_1(R)\cong \ints^{\ds (n-1)}\ds \overline{R}^*$ (where $\ol R$ is the integral closure of $R$);      

\end{example} 

\begin{remark}\label{remark1}

We note here that Examples \ref{ex2} and \ref{ex3} follow from the use of more classical technology.  Let $A\sbe B$ be an inclusion of commutative Noetherian such that $B$ is a module-finite extension of $A$.  Let $I\sbe A$ and $J\sbe B$ be ideals such that $IB\sbe J$.  Set $X = \text{Spec}(A)\setminus \text{Spec}(A/I)$, $Y = \text{Spec}(B)\setminus \text{Spec}(B/J)$, and suppose that the induced morphism of schemes $X\ra Y$ is an isomorphism.  Then Quillen's Localization Theorem (\cite{Q}, Theorem 5) yields long exact sequences 
$$G_i(A/I) \ra G_i(A) \ra G_i(X) \ra G_{i-1}(A/I) \ra\cdots$$
and
$$G_i(B/J) \ra G_i(B) \ra G_i(Y) \ra G_{i-1}(B/J) \ra\cdots$$
Where we note that for a Noetherian scheme $\mathcal S$,  $G_i(\mathcal S)$ is the $i$th Quillen $K$-group of the category of coherent $\mathcal O_\mathcal{S}$-modules.  Now restriction of scalars induces the following commutative diagram 
\begin{diagram}
\cdots &\rTo& G_i(B/J) &\rTo^{\delta_{B/J}} & G_i(B) &\rTo^{\delta_B}& G_i(Y) &\rTo^{\delta_Y} &G_{i-1}(B/J) &\rTo&\cdots\\
&&\dTo_{\varepsilon_{B/J}} && \dTo_{\varepsilon_B} &&\dTo_{\varepsilon_Y} &&\dTo \\ 
\cdots &\rTo& G_i(A/I) &\rTo_{\delta_{A/I}} & G_i(A) &\rTo_{\delta_A} & G_i(X) &\rTo_{\delta_X}& G_{i-1}(A/I)&\rTo&\cdots \\
\end{diagram}
Where we note that $\veps_Y$ is an isomorphism.  Some rather involved but straightforward diagram chasing gives a Mayer-Vietoris-like sequence of $G$-groups that we denote by $(\star)$:    
\begin{equation*}\cdots\ra G_i(B/J) \overset{\alpha}{\ra} G_i(A/I) \oplus G_i(B)  \overset{\beta}{\ra} G_i(A) \overset{\gamma}{\ra}  G_{i-1}(B/J)\ra   \cdots\end{equation*} 
Where $\alpha = \left(\begin{matrix} \veps_{B/J}\\ \delta_{B/J}\end{matrix}\right)$, $\beta = (\delta_{A/I},-\veps_B)$, and $\gamma = \delta_Y \veps_Y^{-1}\delta_A$.  

\vs 

To see how we can recover the claims in Example \ref{ex2} using $(\star)$, let $I = (t^2, t^{2n+1})\sbe k[[t^2, t^{2n+1}]] = A$ and $J = (t) \sbe B = k[[t]]$.  We note that $B$ is a module-finite extension of $A$ and $ \text{Spec}(A)\setminus \text{Spec}(A/I) = \stl (0)\str \cong \text{Spec}(B)\setminus \text{Spec}(B/J) = \stl (0)\str$, so the above requirements are met.  Using $(\star)$, we obtain a long exact sequence 
$$G_i(k) \overset{\alpha}{\ra} G_i(k) \ds G_i(B) \overset{\beta}{\ra} G_i(A) \overset{\gamma}{\ra} G_{i-1}(k)\overset{\alpha'}{\ra} G_{i-1}(k) \ds G_{i-1}(B)$$ 
Where $\alpha' = \left(\begin{matrix} \veps'_{B/J}\\ \delta'_{B/J}\end{matrix}\right)$.  As $I = J\cap A$, the induced map $A/I \ra B/J$ is an isomorphism, so that $\veps_{B/J}$ is an isomorphism.  In particular, we obtain the exact sequence 
$$0\ra G_i(B)/\im(\delta_{B/J}) \ra G_i(A) \overset{\gamma}{\ra} G_{i-1}(k)\overset{\alpha'}{\ra} G_{i-1}(k) \ds G_{i-1}(B)$$ 
Now $\delta_{B/J} = \delta'_{B/J}  = 0$, so that we easily obtain from the above exact sequence $G_i(B) \cong G_i(A)$.  In particular, $G_1(A) = G_1(k[[t^2, t^{2n+1}]] \cong G_1(B) = G_1(k[[t]]) = k[[t]]^*$.  We note, unlike the restriction on the characteristic we encounter using Theorem \ref{thmg1} below, this holds regardless of the characteristic.  

\vs 

To see how we can recover the claims in Example \ref{ex3} using $(\star)$, we let $A = S/(f_1\cdots f_n)$ with $I$ the maximal ideal of $A$ and $B = S/(f_1) \ds\cdots\ds S/(f_n)$ with $J = \mfm_1\ds\cdots\ds \mfm_n$, where $\mfm_i$ is the maximal ideal of the local ring $S/(f_i)$.  It is easy to see that $B$ is a module-finite extension of $A$.  Moreover, $B$ is also the integral closure of $A$ in its total quotient ring.  We also have 
$$X = \text{Spec}(A)\setminus \text{Spec}(A/I) = \stl (f_i)/(f_1\cdots f_n) : 1\leq i\leq n\str$$ 
$$Y = \text{Spec}(B)\setminus \text{Spec}(B/J) = \stl S/(f_1)\ds\cdots\ds \underbrace{0}_i\ds\cdots\ds S/(f_n): 1\leq i\leq n\str $$ 
From the above, it is clear the induced map $X\ra Y$ is given by $(f_i)/(f_1\cdots f_n) \mapsto S/(f_1)\ds\cdots\ds \underbrace{0}_i\ds\cdots\ds S/(f_n)$, hence is clearly an isomorphism.  From $(\star)$, we obtain a long exact sequence 
$$G_i(k)^{\ds n} \overset{\alpha}{\ra} G_i(k) \ds G_i(B) \overset{\beta}{\ra} G_i(A)\overset{\gamma}{\ra} G_{i-1}(k)^{\ds n} \overset{\alpha'}{\ra} G_{i-1}(k)\ds G_{i-1}(B)$$ 
Now the first component of $\alpha$ and $\alpha'$ is the summing map.  Moreover, $\delta^{B/J} = \delta'_{B/J} = 0$ (where $\delta'_{B/J}$ is the second  component of $\alpha'$) so that we obtain an exact sequence 
$$0 \ra G_i(B) \ra G_i(A) \ra G_{i-1}(k)^{\ds (n-1)} \ra 0$$ 
As $f_j\notin (x,y)^2$, $S/(f_j)$ is regular, hence $G_i(B) = K_i(B)$.  Specializing to $i = 1$, we obtain the exact sequence 
$$0 \ra K_1(B) \ra G_1(A) \ra \ints^{\ds (n-1)}\ra 0$$ 
Since the above sequence splits, we obtain $G_1(A) \cong K_1(B) \ds \ints^{n-1}$.  As $K_1(S/(f_j)) = (S/(f_j))^*$, it is easy to see that $K_1(B) = B^*$.  This gives Example \ref{ex3}.  We note that (iii) in the hypothesis of Example \ref{ex3} is not needed, so this is more general.  

\vs 

While we can recover Examples \ref{ex2} and \ref{ex3} from these methods, we find that our work expands on calculations  given in (\cite{v2}, Section 7.3 and Proposition 7.26) and (\cite{H}, Example 10.5).  

\end{remark} 

Lastly, our results that do not follow from our arguments in Remark \ref{remark1} are the following:  

\begin{proposition}\label{thmg2}

Let $k$ be an algebraically closed field of characteristic not two.

\vs 

(a)  If $R = k[[s^2, st, t^2]]$, then $G_1(R)\cong R^*$. 

\vs 

(b)  Suppose $k$ has characteristic zero, $S = k[[x, y]]$, $S' = k[[x, y, u, v]]$, and $R' = S'/(f+uv)S'$, where $f = f_1\cdots f_n\in S = k[[x, y ]]$ is such that 

\vs

    \hspace{.5 cm}(i)  $S/fS$ is an isolated singularity (ie. $(f_i) \neq (f_j)$) 
		
		\vs
		
		\hspace{.5 cm}(ii)  $f_i\notin (x, y)^2$ for all $i$.  
	
	\vs
	
	\hspace{.5 cm}(iii)  $(f_i, f_{i+1}) = (x, y)$.  
	
	\vs
	
Then $G_1(R')\cong \ints^{\ds (n-1)} \ds \left(R'^* \ds k[[w, z]]^{* \ds (n-1)}\right)/\Xi$, with $\Xi$ the subgroup of Definition \ref{defsha} and $w, z$ variables over $k$.  

\end{proposition}

\subsection{The $n$-Auslander-Reiten Matrix}\label{ss16}

Before we can use Theorem \ref{thmg1} to perform the calculation of Examples \ref{ex1}, \ref{ex2}, \ref{ex3}, and prove Proposition \ref{thmg2}, we need to explicitly define the free group $\mathcal H$ occurring in the decomposition of $G_1(R)$ in Theorem \ref{thmg1}.  Our assumptions are as usual and we also require that $R$ is a $k$-algebra and $k$ is algebraically closed of characteristic not two.  We use $L = L_0\ds L_1\ds\cdots\ds L_t$ to denote an $n$-cluster tilting object of $\mcm R$ such that $\Lambda = \End_R(L)^{\text{op}}$ has finite global dimension.  We assume that $L_0, L_1,\ldots, L_t$ are the pairwise non-isomorphic summands of $L$ (each occurs with multiplicity one in the decomposition of $L$) and that $L_ 0 = R$.  Let $I = \stl L_0, L_1,\ldots, L_t\str$ and $I_0 =  I\backslash \stl L_0\str$.  We set $\C = \Ad L$. Recall, for $j > 0$, there is an exact sequence, called the $n$-Auslander-Reiten ending in $L_j$ (see Definition \ref{narhom}):  
$$0\ra C^j_n\ra  \cdots \ra C^j_0 \ra L_j\ra 0$$ 
with $C^j_0, C^j_1,\ldots, C^j_n\in \C$.  Given $N\in \C$, let $\#(i, N)$ be the number of $L_i$-summands ($0\leq i\leq t$) appearing in a decomposition of $N$ into the indecompsables $R$-modules $L_0, L_1,\ldots, L_t$.  Following \cite{v2}, we define a $(t+1)\times t$ integer matrix $T$ whose $ij$-th entry is $\#(i,L_j) + \sum_{u=0}^n(-1)^{u+1}\#(i, C^j_u) = \delta_{ij} + \sum_{u=0}^n(-1)^{u+1} \#(i, C^j_u)$ (note that $T$ has a $0$th row but no $0$th column).  As $G_0(k) = \ints$ and $G_0(\Lambda) = \ints^{\ds t}$, Theorem \ref{viraj} gives us a map $\ints^{\ds t}\ra \ints^{\ds (t+1)}$.  It is shown in (\cite{v2}, Section 7.2) that $T$ defines the map $\ints^{\ds t}\ra \ints^{\ds (t+1)}$ afforded to us by Theorem \ref{viraj}.  We call $T$ the \textit{$n$-Auslander-Reiten matrix} or the \textit{$n$-Auslander-Reiten homomorphism}.  Moreover, this is the same map given in Theorem \ref{k0mod} when $\mcm R$ has a $1$-cluster tilting object.  For our needs, recall Theorem \ref{thmg1} says $G_1(R) = \mathcal H\ds \aut_R(L)_\ab/\Xi$, so that now $\mathcal H = \ker(T)$.   

\vs

We make a useful observation before our computations.  

\begin{lemma}\label{mult}

Let $1\leq i_1 < \cdots < i_h \leq t$ and $L' = L_{i_1}^{\ds q} \ds \cdots \ds L_{i_h}^{\ds q}$ with $q > 0$.  Then for $a\in R^*$, we have $\det_{\Lambda^{\text{op}}}(\wt{a1_{L'}}) = \alpha $, where $\alpha \in (\Lambda^{\text{op}})^*$ is given by $\text{diag}(1_{L_0}, \ldots, a^q1_{L_{i_1}},\ldots, a^q1_{L_{i_h}}, \ldots, 1_{L_t})$

\end{lemma} 

\begin{proof} 

From Remark \ref{situations}, we see that $\wt{a1_{L'}}:  L^{\ds q} \ra L^{\ds q}$ is the map $e1_{L^{\ds q}}$, where $e\in (\Lambda^{\text{op}})^*$ is given by $\text{diag}(1_{L_0}, \ldots, a1_{L_{i_1}},\ldots, a1_{L_{i_h}}, \ldots, 1_{L_t})$.  Now recall the injection $GL_1(\Lambda^{\text{op}}) = (\Lambda^{\text{op}})^* \hookrightarrow GL_q(\Lambda^{\text{op}})$ that takes $\gamma\in(\Lambda^{\text{op}})^*$ to the automorphism $d_1(\gamma) = \text{diag}(\gamma, 1_L,\ldots, 1_L)\in GL_q(\Lambda^{\text{op}})$.  Now 
$$(e1_{L^{\ds q}})^{-1}\cdot d_1(\alpha) = e^{-1}1_{L^{\ds q}}\cdot d_1(\alpha)= \beta_1\cdots \beta_{q-1}$$
where $\beta_u := d_1(e)d_{u+1}(e^{-1})\in GL_q(\Lambda^{\text{op}})$.  Consider the element $\gamma_u := \text{diag}(e, 1,\ldots, 1)$ in $GL_u(\Lambda^\text{op})$.  Then, by slight abuse of notation, the matrix $\delta_u := \text{diag}(\gamma_u, \gamma_u^{-1})$ in $GL_{2u}$ is in the commutator subgroup of $GL_{2u}(\Lambda^\text{op})$ by  (\cite{J}, Corollary 2.1.3).  Thus by (\cite{J}, Proposition 2.1.4), each $\delta_u$ is in the commutator subgroup of $GL(\Lambda^\text{op})$.  In $GL_(\Lambda^\text{op})$, either $\beta_u$ is the image of $\delta_u$ under the injection $GL_{2u}(\Lambda^\text{op})\hookrightarrow GL_q(\Lambda^\text{op})$, or $\delta_u$ is the image of $\beta_u$ under the injection $GL_{q}(\Lambda^\text{op})\hookrightarrow GL_{2u}(\Lambda^\text{op})$.  In either case we see that $\beta_u$ is in the commutator subgroup of $GL(\Lambda^\text{op})$.  Hence $e1_{L^{\ds q}}\equiv d_1(\alpha)$ in $GL(\Lambda^{\text{op}})_\ab$.  

\vs 

Since $\det_{\Lambda^{\text{op}}}:  GL(\Lambda^{\text{op}})_\ab \ra (\Lambda^{\text{op}})^*_\ab$ is the inverse of the isomorphism induced by the map
$$(\Lambda^{\text{op}})^* = GL_1(\Lambda^{\text{op}}) \hookrightarrow GL(\Lambda^{\text{op}})\twohead GL(\Lambda^{\text{op}})_\ab$$ 
We see that $\det_{\Lambda^{\text{op}}}(e1_{L^{\ds q}}) = \alpha$. 
\end{proof}   

\vs 

We note $R$ has finite type if and only if $R$ has a $1$-cluster tilting object $M$.  In this case, $\mcm R = \Ad M$ and $M = M_0\ds M_1\ds\cdots\ds M_t$, with $M_ 0 =R$ and $M_1,\ldots, M_t$ the non-free indecomposable maximal Cohen-Macaulay $R$-modules.  For $j > 0$, we call the $1$-Auslander-Reiten sequence ending in $M_j$ the Auslander-Reiten sequence ending in $M_j$ and the $1$-Auslander-Reiten matrix is referred to as the Auslander-Reiten matrix.  The Auslander-Reiten matrix is a classical invariant and we denote it by $\Upsilon$.  

\vs 

We now make use of Theorem \ref{thmg1} perform the calculations of Examples \ref{ex1}, \ref{ex2}, \ref{ex3}, and prove Proposition \ref{thmg2}.  That is, in the context of Theorem \ref{thmg1}, we must compute the kernel of the $n$-Auslander-Reiten homomorphism and the subgroup $\Xi$ of $\aut_R(L)_\ab$.  In each computation, it will also be clear that $R$ is a $k$-algebra.  

\subsection{Truncated Polynomial Rings in One Variable}\label{sss16}

Our aim here is to utilize Theorem \ref{thmg1} to perform the calculation in Example \ref{ex1}.  That is, if $R = k[x]/x^nk[x]$, then $G_1(R) \cong k^*$.

\begin{proof} 

For $n = 1$, $R = k$, so $G_1(R) = K_1(R) \cong k^*$. 

\vs

We now suppose $n\geq 2$.  Let $\mfm$ denote the maximal ideal $xR$.  By the proof of (\cite{GR}, Theorem 3.3), $R$ has finite type and the indecomposable $R$-modules are given by $R, \mfm, \ldots, \mfm^{n-1}$.  Let $M$ be the $R$-module given by $R\ds \mathfrak m\ds\cdots\ds\mathfrak m^{n-1}$ and denote its endomorphism ring by $E$.  Using (\cite{Y}, Lemma 2.9 ) it is not hard to see the Auslander-Reiten sequences ending in $\mfm^j$ are given by  
\begin{equation} 
0\ra \mathfrak m^j\ra \mathfrak m^{j-1}\ds\mathfrak m^{j+1}\ra \mathfrak m^j\ra 0\hspace{1 cm}(1\leq j\leq n-1)\tag{$\star$}\end{equation} 
Thus for $1\leq j\leq n-2$, $\Upsilon$ has its $j$th column given by $\left(0,\ldots, -1, 2, -1,\ldots, 0\right)^T$, where $-1, 2$ and $-1$ occur in rows $j-1, j$ and $j+1$, respectively.  And the $(n-1)$st column is given by $\left(0,\ldots, 0, -1, 2\right)^T$.  It is easy to see that $\Upsilon$ is injective.  

\vs

We compute the subgroup $\Xi$ of $E^*_\ab$ from Definition \ref{defsha}.  By $(\star)$ and Lemma \ref{mult}, the subgroup $\Xi$ is generated by elements 
$$\xi_{a, j} = (\widetilde{a^21_{\mathfrak m^j}})\cdot (\widetilde{a^{-1}1_{\mathfrak m^{j-1}}\ds a^{-1}1_{\mathfrak m^{j+1}}}) \hspace{1 cm}(1\leq j\leq n-2)$$
$$\xi_{a, n-1} = (\widetilde{a^21_{\mathfrak m^{n-1}}})\cdot (\widetilde{a^{-1}1_{\mathfrak m^{n-2}}})$$

\vs

where $a$ runs over $k^*$.  Again by Lemma \ref{mult}, we have  
$$\xi_{a, j} = \diag(1_R,\ldots, a^{-1}1_{\mfm^{j-1}}, a^21_{\mathfrak m^j}, a^{-1}1_{\mathfrak m^{j+1}}, \ldots, 1_{\mathfrak m^{n-1}})$$ 
$$\xi_{a, n-1} = \diag(1_R,\ldots, \ldots, a^{-1}1_{\mfm^{n-2}}, a^21_{\mfm^{n-1}})$$

By (a) of Proposition \ref{mainab}, there is an isomorphism $E^*_\ab\cong (k^*)^{\ds n}$.  We regard  $\Xi$ as a subgroup of $(k^*)^{\ds n}$ and abuse notation to write  
$$\xi_{a, j} = (1, \ldots, a^{-1}, a^2, a^{-1},\ldots 1)$$
$$\xi_{a, n-1} = (1,\ldots, a^{-1}, a^2)$$
Where $a^{-1}, a^2$ and $a^{-1}$ occur in $\xi_{a, j}$ at positions $j$, $j+1$ and $j+2$, respectively.  Let $\Psi:  (k^*)^{\oplus n}\ra k^*$ be the map such that $\Psi(a_1,\ldots, a_n) = a_1^na_2^{n-1}\cdots a_n$.  Then $\Psi$ is a surjective group homomorphism such that $\Xi \subseteq\text{ker}(\Psi)$.  Let $(a_1,\ldots, a_n)\in \text{ker}(\Psi)$, so that $a_1^na_2^{n-1}\cdots a_n = 1$.  Then \linebreak$(a_1,\ldots, a_n) = \zeta_1\cdots\zeta_{n-1}$, where 
$$\zeta_j = \displaystyle\prod_{v=j}^{n-1}\xi_{a_j^{j-v- 1}, v} $$
Thus $\Psi$ induces an isomorphism $\overline\Psi:  (k^*)^{\oplus n}/\Xi \ra k^*$, hence $G_1(R)\cong k^*$ by Theorem \ref{thmg1}.  

\end{proof} 

\subsection{Singularity of Type $A_{2n}$ in Dimension One}\label{ss26}

The ADE singularity of type $A_{2n}$ is given by the ring $R = k[[t^2, t^{2n+1}]]$.  Here we utilize Theorem \ref{thmg1} to perform the calculation in Example \ref{ex2}.  That is, if the characteristic of $k$ is not $2, 3$ or $5$, then $G_1(R)\cong\overline{R}^* = k[[t]]^*$.  

\begin{proof} 

For $n = 0$, $R = k[[t]]$, a regular local ring, so that $G_1(R)\cong K_1(R) \cong R^* =  k[[t]]^*$ by Quillen's Resolution Theorem (\cite{Q}, \S Theorem  3).  

\vs

We now suppose $n\geq 1$.  Now $R$ has finite type and the indecomposable maximal Cohen-Macaulay $R$-modules are $R_j = k[[t^2, t^{2(n-j)+1}]]$, with $j = 0,\ldots, n$ by (\cite{Y}, Proposition 5.11).  Thus $M$ is the $R$-module $R_0\ds R_1\ds\cdots\ds R_n$ ($R_0 = R$).  Let $E$ be the endomorphism ring of $M$.  By the proof of (\cite{Y}, Proposition 5.11), the Auslander-Reiten sequence ending in $R_j$ is 
$$\begin{array}{cc}
0\ra R_j\ra R_{j-1}\ds R_{j+1} \ra R_j\ra 0 & (1\leq j < n)\\ 
\\
0\ra R_n\ra R_{n-1}\ds R_n \ra R_n\ra 0 & 
\end{array} $$
Thus the Auslander-Reiten matrix $\Upsilon$, for $1\leq j\leq n-1$, has $j$th column given by $\left(0,\ldots, -1, 2, -1, \ldots, 0\right)^T$, with $-1, 2$ and $-1$ occur in rows $j-1, j$ and $j+1$, respectively.  The $n$th column is given by $\left(0,\ldots, 0, -1, 1\right)^T$.  Now $\Upsilon$ is clearly injective, hence $G_1(R)\cong E^*_\ab/\Xi$ by Theorem \ref{thmg1}.  We calculate the subgroup $\Xi$ occurring of Definition \ref{defsha}.  By Lemma \ref{mult}, the subgroup $\Xi$ is generated by the elements 
$$\xi_{a, j} = \wt{a^21_{R_j}} \cdot \wt{a^{-1}1_{R_{j-1}}\ds a^{-1}1_{R_{j+1}}} \hspace{1 cm}(1\leq j < n)$$
$$\xi_{a, n} = \wt{a^21_{R_n}} \cdot \wt{a^{-1}1_{R_{n-1}}\ds a^{-1}1_{R_n}}$$
Where $a$ runs over $k^*$.  We abuse notation and regard $\Xi$ as a subgroup of $(k^*)^{\ds (n+1)}$.  We compute compute $(k^*)^{\oplus (n+1)}/\Xi$, viewing the elements of $\Xi$ as a row vectors in $(k^*)^{\oplus (n+1)}$.   Hence the elements that generate $\Xi$ are given by 
$$\xi_{a, j} = (1,\ldots, a^{-1}, a^2, a^{-1},\ldots, 1) \hspace{1 cm}(1\leq j < n)$$
$$\xi_{a, n} = (1,\ldots, a^{-1}, a)$$
Where $a^{-1}, a^2$ and $a^{-1}$ occur in positions $j, j+1$ and $j+2$ for $1\leq j < n$.  Let $\chi:  (k^*)^{\oplus (n+1)}\ra k^*$ be given by $\chi(a_1,\ldots, a_{n+1}) = a_1\cdots a_{n+1}$.  Then $\text{ker}(\chi)$ is generated by elements of the form $(a_1,\ldots, a_{n+1})$ such that  
$$(a_1,\ldots, a_{n+1}) = (a_2^{-1}, a_2, 1,\ldots, 1)(a_3^{-1}, 1, a_3, 1, \ldots, 1)\cdots (a_{n+1}^{-1},1, 1, \ldots, a_{n+1})$$ 
We show $\Xi = \text{ker}(\chi)$.  Obviously, $\Xi\subseteq \text{ker}(\chi)$.  For the converse, it suffices to show the elements $\zeta_{a, j}= (a^{-1},1 \ldots, a,\ldots, 1)$, where $a$ is in the $j$th position and $2\leq j\leq n+1$, are in $\Xi$.  Indeed, note that $\zeta_{2, a}= \xi_{a, 1}\xi_{a, 2}\cdots\xi_{a, n}$ and for $j > 2$, we have $\zeta_{a, j}  = \zeta_{a, j-1}\xi_{a, j-1}\xi_{a, j}\cdots \xi_{a, n}$.  Thus $\text{ker}(\chi) = \Xi$ as needed.  
\vs
Combining the above and using (b) of Proposition \ref{mainab}, we have 
$$G_1(R)\cong (k^*)^{\ds (n+1)}/\Xi \ds \left(1+tk[[t]]\right) \cong k^* \ds \left(1+tk[[t]]\right) \cong k[[t]]^*$$  

\end{proof}

\subsection{Reduced Hypersurface Singularities in Dimension One}\label{sss46}

Our aim here is use Theorem \ref{thmg1} to perform the calculation in Example \ref{ex3}.  We recall Example \ref{ex3}.  We let $S = k[[x, y]]$,  $f_1,\ldots, f_n\in (x,y)\backslash (x, y)^2$, with $f_i$ irreducible, $f = f_1\cdots f_n$, $R = S/fS$ is an isolated singularity (i.e. $f_iS \neq f_jS$), and $(f_i, f_{i+1}) = (x, y)$.  Then $G_1(R)\cong \ints^{\ds (n-1)} \ds \ol{R}^*$.  
\begin{proof} 
Now $L = S_1\ds\cdots\ds S_n$, with $S_i = S/(f_1\cdots f_i)$, is a $2$-cluster tilting object in $\mcm R$.  In order to compute $G_1(R)$, we need to understand the structure of the $2$-Auslander-Reiten sequences in $\C = \Ad L$.  By (\cite{i}, Proof of Theorem 4.11) the $2$-Auslander-Reiten sequences ending in $S_j$ are 
$$0\ra S_j\ra S_{j+1}\ds S_{j-1}\ra S_{j+1}\ds S_{j-1}\ra S_j\ra 0\hspace{1 cm}(1\leq j < n)$$ 
From this and Lemma \ref{mult} it is clear that the subgroup $\Xi$ of $\aut_R(L)_\ab$ is trivial.  Moreover, from this, it is easy to see that the $2$-Auslander-Reiten matrix $T:  \ints^{\ds (n-1)}\ra \ints^{\ds n}$ is zero.  Thus by Theorem \ref{thmg1} and (d) of Proposition \ref{mainab}
$$G_1(R) \cong \ker(T) \ds \aut_R(L)_\ab\cong  \ints^{\ds (n-1)} \ds \overline{R}^*$$ 
\end{proof}

\subsection{Singularity of Type $A_1$ in Dimension Two}\label{sss36}

Our aim here is to prove (a) of Proposition \ref{thmg2}.  That is, if $R$ is the ring $k[[s^2, st, t^2]]$ then $G_1(R)\cong R^*$.  

\begin{proof} 

By (\cite{GR}, Example 5.25 and 13.21) $R$ has finite type and the indecomposable maximal Cohen-Macaulay $R$-modules are $R$ and $I = (s^2, st)$.  Moreover, the Auslander-Reiten sequence ending in $I$ is given by 
$$0\ra I\ra R^2 \ra I\ra 0$$  
Set $M = R\ds I$ and let $E$ be its endomorphism ring.  

\vs  

An easy calculation shows that the Auslander-Reiten homomorphism $\Upsilon:  \ints \ra \ints^{\ds 2}$ is injective.  Now $\Xi$ is the subgroup of $E^*_\ab$ generated by the elements 
$$\wt{a1_I}\cdot \text{det}_E(\wt{a1_{R^2}})^{-1}\cdot\wt{a1_I} = \wt{a^21_I}\cdot \text{det}_E(\wt{a1_{R^2}})^{-1}\hspace{.2 cm}(a\in k^*)$$
The automorphism of $M$, $\wt{a^21_I}$, is given by $\text{diag}(1_R, a^21_I)$.  Using Lemma \ref{mult}, $\text{det}_E(\wt{a1_{R^2}})$ is the image of the automorphism $\text{diag}(a^21_R, 1_I)$ in $E^*_\ab$.  Thus $\Xi$ is the subgroup of $E^*_\ab$ generated by the elements  
$$\text{diag}(1_R, a^21_I)\cdot \text{diag}(a^{-2}1_R, 1_I) =  \text{diag}(a^{-2}1_R, a^21_I)$$
As groups, $\Xi\cong k^{*2} = \left\{a^2 : a\in k^*\right\}$.  Since $k$ is algebraically closed, $k^{*2} = k^*$.  Using (c) of Proposition \ref{mainab}, we have $E^*_\ab\cong k^*\ds R^*$, hence $E^*_\ab/\Xi\cong R^*$. Thus $G_1(R) \cong R^*$ by Theorem \ref{thmg1}, since $\Upsilon$ is injective.  
\end{proof}  

\subsection{Reduced Hypersurface Singularities in Dimension Three}\label{sss56}
 
Our aim here is to prove (b) of Proposition \ref{thmg2}.  We recall (b).  If  $S' = k[[x, y, u, v]]$, $R' = S'/(f+uv)S'$, where $f = f_1\cdots f_n$ with $f_i\in (x, y)\backslash (x, y)^2\subseteq S= k[[x, y]]$ are such that then $G_1(R') \cong \ints^{\ds (n-1)} \ds \left(R'^* \ds k[[w, z]]^{* \ds (n-1)}\right)/\Xi$, with $\Xi$ the subgroup from Theorem \ref{thmg1} and $w, z$ variables over $k$.

\begin{proof} 
If $U_i = (u, f_1\cdots f_i)$, then $L = U_1\ds\cdots\ds U_n$ is a $2$-cluster tilting object.  Then by (\cite{v2}, Proposition 7.28), the $2$-Auslander-Reiten matrix $T$ is zero.  By (e) of Proposition \ref{mainab}, $\aut_{R'}(L)_\ab \cong R'^* \ds (k[[w,z]])^{* \ds (n-1)}$ ($w$ and $z$ variables over $k$), thus Theorem \ref{thmg1} yields 
$$G_1(R') \cong \ints^{\ds (n-1)} \ds \left(R'^* \ds k[[w, z]]^{* \ds (n-1)}\right)/\Xi$$
Where $\Xi$ is the subgroup of $R'^* \ds k[[w, z]]^{*\ds (n-1)}$ of Definition \ref{defsha}. 

\end{proof} 


\section{Discussion}\label{discussion}

It is of interest to note that in the calculations of $G_1(R)$ for $R$ of positive dimension, either $G_1(R)\cong \overline{R}^*$ ($\overline{R}$ is the integral closure of $R$ in its total quotient ring), or $G_1(R)$ contains $\overline{R}^*$ a direct summand.  Our methods were ad hoc and tailored specifically to each singularity via the calculation of the group $\aut_R(L)_\ab $, so a deeper look into the relationship between $\Lambda = \End_R(L)^{\text{op}}$ and $\overline R$ could shed some light on the structure of $G_1(R)$ for hypersurface singularities.  

\vs

In fact, the key to the relationship seems to be understanding the relationship between the derived categories of $\Mod \End_R(L)^\text{op}$ and $\Mod \overline{R}$.  Indeed, in \cite{SD}, it is shown that if $A$ and $B$ are Noetherian (not necessarily commutative) rings whose derived categories are equivalent as triangulated categories, then there is an isomorphism $G_i(A)\cong G_i(B)$ for $i\geq 0$.  Of course, one should not expect an equivalence of the derived categories of $\Mod \End_R(L)^\text{op} $ and $\Mod\overline{R}$ since our examples (see Proposition \ref{mainab}) indicate for positive-dimensional rings that $G_1(\End_R(L)^\text{op}) \cong \aut_R(L)_\ab$ only contains $\overline{R}^*$ as a direct summand.  Moreover, it may also be too much to ask that $G_1(\overline{R})$ is a direct summand of $G_1(\End_R(L)^\text{op})$, as $G_1(\overline R)$ is not always isomorphic to $\overline{R}^*$.  However, if $R$ is a reduced one-dimensional local Noetherian ring, then $\overline{R} = \overline{R/\mfp_1}\times\cdots \times \overline{R/\mfp_s}$, where the $\mfp_j$ are the minimal primes of $R$ and each ring $\overline{R/\mfp_j}$ is a semilocal principal ideal domain.  In this situation
$$G_1(\overline{R})\cong G_1(\overline{R/\mfp_1})\times\cdots\times G_1(\overline{R/\mfp_s})$$
Now $\overline{R/\mfp_j}$ is semilocal and has finite global dimension, hence if $R$ is an algebra over a field $k$ with $\text{char}(k)\neq 2$, then Quillen's Resolution Theorem (\cite{Q}, \S Theorem  3), (\cite{VS}, Corollary 2.6 and Theorem  5.1), and (\cite{L}, Theorem  2) show there are isomorphisms 
$$G_1(\overline{R/\mfp_j})\cong K_1(\overline{R/\mfp_j}) \cong K^C_1(\overline{R/\mfp_j}) = (\overline{R/\mfp_j})^*$$
Thus $G_1(\overline{R})\cong \overline{R}^*$ in this case.  Nevertheless, we conjecture that if $R$ satisfies the hypotheses of Theorem \ref{thmg1} and has positive dimension, then $\aut_R(L)_\ab/\Xi \cong \overline{R}^*$ and hence $G_1(R)$ is isomorphic to the direct sum of the kernel of the $n$-Auslander-Reiten homomorphism and $\overline{R}^*$.  

\section*{Acknowledgments}
The author would like to thank Hailong Dao and Jeanne Duflot for their useful comments in the preparation of this manuscript.  We would also like to thank the anonymous referee for greatly improving the quality of this manuscript.


\begin{thebibliography}{}

\bibitem{B} Bass, H.,  \textit{Algebraic K-Theory}, W.A. Benjamin, Inc.  New York, New York (1968). 

\bibitem{hc} Dao, H. and Huneke, C., Vanishing of Ext, Cluster Tilting modules and Finite Global Dimension of Endomorphism Rings, \textit{American Journal of Mathematics}, 135, 561-578 (2013).  

\bibitem{dfi1} Dao, H., Faber, E., Ingalls, C., Noncommutative (Crepant) Desingularizations and the Global Spectrum of Commutative Rings.  \textit{Algebras and Representation Theory}, 18, 633-664 (2015).  

\bibitem{dfi2} B. Doherty, E. Faber, C. Ingalls, \textit{Computing Global Dimension of Endomorphism Rings via Ladders}, \textit{Journal of Algebra}, 458, 307-350 (2016).  

\bibitem{SD} D. Dugger and B. Shipley, K-theory and Derived Equivalences, \textit{Duke Math Journal}, 124, 587-617 (2004).  

\bibitem{ger} Gersten, S.M., \textit{Algebraic K-Theory I. Proceedings of the Conference Held at the Seattle Research Center of Battelle Memorial Institute, August 28 - September 8, 1972}, Higher K-theory of Rings, 3-42, Springer-Verlag, Berlin (1973).   

\bibitem{H} Holm, H., K-Groups for Rings of Finite Cohen-Macaulay Type, \textit{Forum Mathematicum}, 27, 2413-2452 (2015).  

\bibitem{i} Iyama, O. \textit{Trends in Representation Theory of Algebras and Related Topics}, Auslander-Reiten Theory Revisited, European Mathematical Society, Z\"{u}rich, 349-398 (2008).    

\bibitem{i2} Iyama, O., Higher-Dimensional Auslander Reiten Theory on Maximal Orthogonal Subcategories, \textit{Advances in Mathematics}, 210, 22-50 (2007).     

\bibitem{Lam} Lam, T.Y.,  \textit{A First Course in Noncommutative Rings}, Springer-Verlag, New York (2001).   

\bibitem{GR} Leuschke, G.J. and Wiegand, R., \textit{Cohen-Macaulay Representations}, American Mathematical Society, Providence, RI (2012).   

\bibitem{G}  Leuschke, G.J., Endomorphism Rings of Finite Global Dimension, \textit{Canadian Journal of Mathematics}, 59, 332-342 (2007).  

\bibitem{v2} Navkal, V (2013), \textit{$K'$-theory of a Cohen-Macaulay Local Ring with $n$-Cluster Tilting Object} (Doctoral dissertation), Retrieved from ProQuest Dissertations and Theses (Order No. 3563356)

\bibitem{pg}  Peng, Y. and Guo, X, The $K_1$-group of tiled orders, \textit{Communications in Algebra}, 41, 3739-3744 (2013).  

\bibitem{Q}  Quillen, D., \textit{Algebraic K-Theory I. Proceedings of the Conference Held at the Seattle Research Center of Battelle Memorial Institute, August 28 - September 8, 1972}, Higher Algebraic K-theory I, 85-147, Springer-Verlag, Berlin (1973).   

\bibitem{J} Rosenberg, J.,  \textit{Algebraic K-theory and its Applications}, Springer-Verlag, New York (1994).     

\bibitem{sher} Sherman, C.,  \textit{Algebraic K-Theory. Proceedings of a Conference Held at Oberwolfach, June 1980.  Part 1}, Group Representations and Algebraic K-theory, 208-243, Springer-Verlag, Berlin (1982).   

\bibitem{VS} Srinivas, V.,  \textit{Algebraic K-theory}, Birkh\"{a}user, Boston, MA (1996).  

\bibitem{int} Swanson, I. and Huneke, C.,  \textit{Integral Closure of Ideals, Rings and Modules}, Cambridge University Press, Cambridge (2006).  

\bibitem{L}  Vaserstein, L.N.,  On the Whitehead Determinant for Semi-local Rings, \textit{J. of Algebra}, 283, 690-699 (2005).  

\bibitem{Y}  Yoshino, Y., \textit{Cohen-Macaulay Modules Over Cohen-Macaulay Rings}, Cambridge University Press, Cambridge (1990).  

\end{thebibliography}
\end{document}